\documentclass[10pt]{article}
\usepackage[utf8]{inputenc}
\usepackage[T1]{fontenc}
\usepackage{lmodern}
\usepackage{amsmath}
\usepackage{amsfonts}
\usepackage{amssymb}
\usepackage{amsthm}
\usepackage{cases}
\usepackage{mathrsfs}
\usepackage{url}
\usepackage{authblk}
\usepackage{color}
\usepackage{geometry}
\usepackage{indentfirst}
\usepackage{hyperref}
\hypersetup{
	colorlinks=true,       
	linkcolor=blue,        
	citecolor=blue,        
	urlcolor=blue,         
	pdfborder={0 0 0},     
	pdfstartview=FitH,     
	pdfauthor={Your Name}, 
	pdftitle={Sample PDF}  
}

\allowdisplaybreaks[4]

\theoremstyle{plain}
\newtheorem{thm}{Theorem}[section]

\newtheorem{pro}[thm]{Problem}
\newtheorem{lem}[thm]{Lemma}

\theoremstyle{definition}

\newtheorem{ass}{Assumption}[section]

\newcommand{\eps}{\varepsilon}

\makeatletter\@addtoreset{equation}{section} \makeatother

\begin{document}		
	\title{ Stochastic Maximum Principle for Fully Coupled Nonlinear FBS$\Delta$Es under Generalized Monotonicity and LQ Control Applications
		\thanks{Qingxin Meng was supported by the National
			Natural Science Foundation of China ( No.12271158) and  the Key Projects of Natural Science Foundation of Zhejiang Province (No. LZ22A010005).}}

	\date{}
	
	\author[a]{Zhipeng Niu}
	\author[b]{Jun Moon}
	\author[a]{Qingxin Meng\footnote{Corresponding author.
			\authorcr
			\indent E-mail address:zpengniu@163.com(Z. Niu),
junmoon@hanyang.ac.kr          mqx@zjhu.edu.cn (Q.Meng) }}

	\affil[a]{\small{Department of Mathematical Sciences, Huzhou University, Zhejiang 313000,PR  China}}
	
	\affil[b]{\small{Department of Electrical
Engineering, Hanyang University,
Seoul, 04763, South Korea}}

	\maketitle

	\begin{abstract}
This paper investigates the optimal control problem for a class of nonlinear fully coupled forward-backward stochastic difference equations (FBS$\Delta$Es). Under the convexity assumption of the control domain, we establish a variational formula for the cost functional involving the Hamiltonian system and adjoint equations, deriving both necessary and sufficient optimality conditions via the Pontryagin maximum principle. Innovatively, we employ a generalized monotonicity framework to ensure the existence and uniqueness of solutions for nonlinear systems and directly derive variational inequalities through the convexity properties of the Hamiltonian function, simplifying the analysis of fully coupled systems. As an application, we formulate a linear-quadratic (LQ) optimal control problem inspired by energy storage scheduling (a real-world example) to demonstrate the effectiveness of our theoretical results. The study reveals that discrete-time FBS$\Delta$Es models offer significant computational advantages for practical systems with future-dependent constraints, such as power dispatch and financial decision-making, providing a new theoretical foundation for high-dimensional optimal control problems.

	\end{abstract}

	
	\section{ Introduction}
In the expansive field of optimal control theory, the stochastic maximum principle (SMP) stands as a cornerstone methodology, offering a robust theoretical framework for addressing complex control problems under uncertainty. Its introduction marked a significant leap forward, extending the scope of control theory from deterministic to stochastic systems, thereby broadening the range of applications and enabling the control of more realistic, noise-driven systems.

The roots of the stochastic maximum principle trace back to the deterministic maximum principle, introduced by Pontryagin, which has long been central to control theory. The pioneering works of Kushner  \cite{ref-1} and Bismut \cite{ref-2} successfully extended this principle to stochastic environments, laying the foundation for the rapid development of stochastic control theory. At its core, the stochastic maximum principle transforms optimal control problems into the task of solving forward-backward stochastic differential equations (FBSDEs), enabling the derivation of optimal control strategies through the Hamiltonian function.

The exploration of the stochastic maximum principle in forward stochastic differential equations (FSDEs) has seen substantial progress. Bismut's seminal work in 1978 \cite{ref-3} provided a comprehensive review of stochastic optimal control theory, establishing multiple maximum principles and investigating their duality applications, thus setting the stage for future developments. In 1990, Peng \cite{ref-4} made a major breakthrough by deriving the maximum principle for nonlinear optimal control problems under general conditions, addressing challenges such as non-convex control domains and control-dependent diffusion terms, significantly expanding the applicability of the stochastic maximum principle. In 1993, Peng \cite{ref-5} further advanced the field by deriving the maximum principle for forward-backward stochastic differential equations (FBSDEs) under the assumption of a convex control domain, inspiring extensive research into increasingly complex systems.  For a truly all - encompassing review that delves into every aspect of this body of research, one can refer to a multitude of sources including, but not limited to, \cite{ref-7,ref-8,ref-9,ref-10,ref-11, ref-111}, etc.

The fusion of forward and backward stochastic differential equations led to the development of fully coupled forward-backward stochastic differential equations (FBSDEs), where the initial and terminal states of the system are interdependent. For a long time, proving the existence and uniqueness of solutions for such systems posed a significant challenge, hindering progress in the corresponding stochastic maximum principle. In 1995, Hu and Peng \cite{ref-12} made a landmark contribution by employing continuation methods to prove, for the first time, the existence and uniqueness of solutions for fully coupled FBSDEs. This work was later refined by Peng and Wu \cite{ref-13} and Yong \cite{ref-14}. More recently, Yu \cite{ref-15}, inspired by stochastic linear-quadratic problems, introduced a generalized monotonicity framework and successfully demonstrated the existence and uniqueness of solutions for a class of FBSDEs with coupled initial and terminal values, reinvigorating the field.

The literature on the stochastic maximum principle for fully coupled forward-backward stochastic differential equations (FBSDEs) in continuous time has grown increasingly rich, encompassing more complex models and diverse applications. The seminal work by Wu in 1998 \cite{ref-16} pioneered the study of optimal control systems for fully coupled FBSDEs, establishing the maximum principle under the assumption of a convex control domain. In 2006, Shi and Wu \cite{ref-17} extended these findings to systems where the diffusion term is independent of the control, proving the corresponding maximum principle without requiring convexity assumptions on the control domain. Subsequent advancements include Huang and Shi's \cite{ref-18} exploration of time-delay systems and Hu, Ji, and Xue's \cite{ref-19} innovative decoupling method. More recently, Wang et al. \cite{ref-20} investigated mean-field systems under partial observation, deriving variational formulas and establishing both necessary and sufficient optimality conditions. For further significant contributions, refer to the comprehensive literature \cite{ref-131,ref-141,ref-151,ref-161} in this field, among others.

 {Alongside the development of continuous-time systems, discrete-time stochastic difference equations have garnered increasing attention.  Discrete-time systems are more representative of real-world applications, as even continuous-time systems are often discretized for analysis and computational control. The rise of deep learning and the growing need for efficient numerical solutions further emphasize the importance of discrete-time systems. }Compared to continuous-time FBSDEs, fully coupled stochastic difference equations (FBS$\Delta$Es) offer significant computational advantages, making them more suitable for numerical computation and modeling in practical applications, especially in fields like finance, control, and optimization. For instance, discrete stochastic systems with time-varying delays have been shown to admit delay-dependent mean-square stability criteria formulated via linear matrix inequalities (LMIs), enabling efficient computational verification and controller synthesis \cite{gao2007discrete}. Moreover, for nonlinear discrete-time stochastic systems with time-varying delays, finite-time stochastic stability (FTSS) can be guaranteed using discrete-time Lyapunov–Krasovskii functionals and LMI-based piecewise state-feedback controllers, as validated by numerical simulations \cite{liu2024finite}. Furthermore, in the context of epidemiological modeling, discrete-time stochastic approaches—such as those applied to tuberculosis dynamics with relapse and Lévy noise—support robust existence and extinction/persistence results alongside discrete numerical approximation, demonstrating the rich modeling capability of discrete stochastic frameworks \cite{benazzouz2024discontinuous}.

In addition to its theoretical appeal, this study is also driven by strong practical motivation. In real-world applications such as energy storage scheduling, decision-makers must determine the present charging or discharging strategy based on uncertain future energy prices, load demands, and regulatory constraints. These decisions inherently involve anticipating future states and costs, leading to time-inconsistent objectives that require bidirectional temporal reasoning.
Mathematically, this motivates the formulation of \emph{fully coupled} FBS$\Delta$Es, in which the forward state variable \( x_k \) (e.g., energy level) depends not only on the current control \( u_k \) but also on the future backward variable \( y_{k+1} \), which typically encodes future costs or constraints. Conversely, the backward variable \( y_k \) is influenced by the current state \( x_k \) and control \( u_k \), resulting in a bidirectionally dependent dynamic, such as:
$$
x_{k+1} = f(k, x_k, u_k, y_{k+1}), \quad
y_k = g(k, x_k, u_k, y_{k+1}).
$$
This fully coupled structure is realistic and essential in modeling systems with future-dependent constraints, such as power systems, financial portfolio optimization, and dynamic resource allocation.
A more detailed background introduction is provided in Section 4.

Despite its practical significance, the stochastic maximum principle (SMP) for general nonlinear fully coupled FBS$\Delta$Es in discrete time remains relatively underdeveloped. Most existing results rely on linear assumptions or partial coupling.

In recent years, theoretical and applied research on FBS$\Delta$Es has gained widespread attention, particularly in China. Notable contributions include the work of Zhang Huanshui, Xu Juanjuan, and their collaborators \cite{ref-21,ref-22,ref-23}, who studied the solution theory of linear FBS$\Delta$Es using decoupling theory for linear-quadratic optimal control problems. Additionally, Ji  and  Liu \cite{ref-24} established the existence and uniqueness of solutions for nonlinear FBS$\Delta$Es under traditional monotonicity assumptions based on martingale representation theory.

In modern applied mathematics and control theory, discrete-time stochastic control problems occupy a critical position, with applications in financial engineering, systems engineering, and communication technology. In particular, mean-field-type discrete-time stochastic control problems have recently attracted significant interest due to their ability to model the coupling relationships between individual behaviors and the average state of a group. These problems provide a more realistic theoretical framework for modeling complex systems, highlighting the dynamic behavior of multi-agent systems.

Initial research on the discrete-time maximum principle laid the foundation for further development in this field. Early work proposed basic maximum principle concepts for simple discrete-time models, which spurred more comprehensive studies.
Lin and Zhang \cite{ref-26} constructed new models for discrete-time stochastic systems with multiplicative noise, proposing corresponding maximum principles and broadening the research scope.
Wu and Zhang \cite{ref-27} derived necessary and sufficient optimality conditions for discrete-time stochastic control problems with convex control domains, including discrete-time stochastic games, thereby establishing a rigorous theoretical framework.
 Dong, Nie, and Wu \cite{ref-28} made significant strides in studying mean-field-type discrete-time stochastic control problems. These problems present numerous theoretical challenges, especially due to the presence of mean-field terms.
 Their theoretical analysis and case studies demonstrated the practical effectiveness of the derived expression in solving mean-field-type discrete-time stochastic control problems, providing a new theoretical tool for further research.
 Dong, Nie, and Wu (2023) \cite{ref-32}  studied the maximum principle for discrete - time stochastic optimal control problems with delay, overcame difficulties from delayed responses via an adjoint operator technique to derive adjoint equations and duality relations, and obtained the maximum principle for optimal control.
 In 2024, Ji and Zhang \cite{ref-29} analyzed the properties of solutions to infinite-horizon backward stochastic difference equations and proposed related stochastic recursive control strategies. In the same year, Song and Yao derived second-order necessary optimality conditions for discrete-time stochastic systems, offering more precise criteria for optimal solutions. Ahmadova and Mahmudov \cite{ref-30} presented a new formulation of the stochastic maximum principle for discrete-time mean-field optimal control problems, proving its validity through experiments.  In 2024, Song and Wu \cite{ref-31} proposed a general theoretical framework for discrete-time stochastic control problems and demonstrated its generality and effectiveness through case studies. These advancements have greatly enriched the theoretical framework of the maximum principle for discrete-time stochastic systems. The study of discrete-time forward-backward systems holds significant theoretical importance and practical value, particularly when combined with deep learning methods for high-dimensional numerical solutions \cite{ref-33}.

Building on the extensive body of research in stochastic control theory, this paper focuses on the maximum principle for optimal control of fully coupled forward-backward stochastic difference equations (FBS$\Delta$Es). Our work extends the contributions of Niu et al. \cite{ref-34}, who investigated the solvability of nonlinear fully coupled FBS$\Delta$Es under Yu’s monotonicity framework. The novel contribution of this paper lies in the study of a new class of nonlinear systems where both the state process and the initial and terminal values are fully coupled, a feature not thoroughly explored in prior research. Under the assumption of a convex control domain, we derive the variational formula for the objective function using the Hamiltonian function and the adjoint system, establishing the necessary conditions for the maximum principle. Furthermore, by imposing appropriate conditions on the initial or terminal values, we prove the sufficiency of the maximum principle, thereby providing a comprehensive theoretical framework for optimal control in discrete-time systems.

It is important to acknowledge the pioneering contributions of  Ji and Liu \cite{ref-25}, who established the first comprehensive framework for the stochastic maximum principle in both partially and fully coupled forward-backward stochastic difference equations (FBS$\Delta$Es). Their seminal work delivers critical theoretical foundations via several key innovations: a unified analytical approach for partially coupled systems with tripartite solutions \((X,Y,Z)\) and fully coupled systems using extended quadrivariate solutions \((X,Y,Z,N)\) (where \(N\) denotes an orthogonal martingale process); well-posedness results (existence and uniqueness of solutions) under traditional monotonicity conditions to overcome limitations of continuous-time frameworks; novel discrete-time variational techniques to handle cross-term interactions in iterative systems and resolve significant computational challenges; and successful applications to practical domains including discrete nonlinear expectations, constrained option pricing, and time-consistent risk management. { However, our work differs significantly from theirs in several key aspects. First, the form of the backward equation in the fully coupled forward-backward stochastic difference equations is different: their solution consists of three components ($Y$, $Z$, $N$), whereas our solution comprises only one component. Second, methodologically, our research introduces two key innovations:
first, we employ a generalized monotonicity condition adapted and extended from Yu’s continuous-time framework to fit discrete-time fully coupled structures, which not only ensures the existence and uniqueness of solutions for nonlinear FB$\Delta$Es (without relying on restrictive traditional pointwise monotonicity assumptions for state variables) — but also features a flexible structure that naturally encompasses and enables the solution of diverse linear-quadratic (LQ) optimal control problems. Specifically, common discrete-time LQ scenarios (e.g., LQ control with initial/terminal state constraints, forward-backward coupled LQ systems, LQ problems with multiplicative stochastic perturbations, and LQ control for networked discrete systems) can all be subsumed under the scope of this condition, without the need to introduce additional ad-hoc assumptions for individual LQ cases. This unifies the solvability analysis of nonlinear FB$\Delta$Es and the solution of LQ control problems into a single theoretical framework: we first verify that the target LQ system satisfies the generalized monotonicity condition, then directly derive the optimal control strategy based on the system’s solvability result under this condition; second, we derive the variational inequality directly through the Hamiltonian function rather than perturbing the state equations directly. This approach leverages the convexity properties of the Hamiltonian function, simplifying the analysis and enhancing the applicability of our results to complex nonlinear control problems. Our methodology focuses on the Hamiltonian framework, while their approach is more traditional, relying on classical variational methods. These methodological distinctions set our work apart from previous studies that rely on classical variational methods and traditional monotonicity conditions.}
	
The remainder of the paper is organized as follows: Section 2 introduces the necessary notation and presents the optimal control system for the FBS$\Delta$Es under study, including lemmas on the existence and uniqueness of solutions. The adjoint equations and Hamiltonian system are also introduced as the foundation for the subsequent analysis.
Section 3 derives the variational formula for the objective function and establishes the necessary and sufficient conditions for the maximum principle.
Section~4 demonstrates the practical applicability of our theoretical results by reformulating a real-world energy storage scheduling problem as a linear-quadratic (LQ) optimal control problem governed by fully coupled FBS$\Delta$Es. The example illustrates how the discrete-time stochastic maximum principle can be applied to derive explicit control strategies under uncertainty.
Section 5 concludes the paper by summarizing the research findings and suggesting directions for future work.

This paper aims to contribute to the growing body of knowledge in stochastic control theory, offering new insights and tools for addressing complex discrete-time optimal control problems.

	\section{Notations and Preliminaries }\label{sec:2}
	Let $N$ be a given positive integer representing the time horizon.
	$\mathbb{T}$ , $\overline{\mathbb{T}}$  and  \b{${\mathbb T}$}   denote the sets  $\{0,1, \cdots, N-1\}$,  $\{0,1,2, \ldots, N\}$ and $\{1,2, \ldots, N\}$ respectively.
	Let $(\Omega, \mathcal{F}, \mathbb{F}, \mathbb{P})$ be a complete filtered probability space
	with a filtration $\mathbb{F}=\left\{\mathcal{F}_k : k=0, \cdots, N-1\right\}$. Assume that $\left\{w_k\right\}_{k\in \mathbb T}$ is a $\mathbb{R}$-value martingale difference sequence defined on a probability space $(\Omega, \mathcal{F}, P)$, i.e.
	$$
	\mathbb{E}\left[w_{k+1} \mid \mathcal{F}_k\right]=0, \quad \mathbb{E}\left[w_{k+1}^2 \mid \mathcal{F}_k\right]=1.
	$$
	In this paper, we assume that $\mathcal{F}_k$ is the $\sigma$-algebra generated by $\left\{x_0, w_l, l=\right.$ $0,1, \ldots, k\}$. For convenience, $\mathcal{F}_{-1}$ denotes $\left \{ \emptyset,\Omega  \right \}$.

	Let $\mathbb{R}^n$ be the $n$-dimensional Euclidean space with the norm $|\cdot|$ and the inner product $\langle\cdot,\cdot\rangle$. Let $\mathbb{S}^n$ be the set of all symmetric matrices in $\mathbb{R}^{n\times n}$. Let $\mathbb{R}^{n\times m}$ be the collection of all $n\times m$ matrices with the norm $|A|=\sqrt{\textrm{tr}(AA^\top)}$, for $\forall A\in \mathbb{R}^{n\times m}$ and the inner product:
	\begin{equation}
		\left\langle A,B\right\rangle = \textrm{tr}(AB^\top),\quad A,B\in \mathbb{R}^{n\times m}.\nonumber
	\end{equation}

	Let $\mathbb H$ be a Hilbert space with norm $\|\cdot\|_\mathbb H$, then we introduce some notations as follows:

	

	$\bullet$ $L^{2} _{\mathcal{F}_{N-1} }(\Omega ; \mathbb H)$: the set of all $\mathbb H$-valued  $\mathcal{F}_{N-1}$-measurable random variables $\xi$ satisfying
	
	\begin{equation}
		\|\xi\|_{L^{2} _{\mathcal{F}_{N-1} }(\Omega ;\mathbb{H})} :=\Big[\mathbb{E}\|\xi \|_\mathbb H^{2} \Big]^{\frac{1}{2}}<\infty .\nonumber
	\end{equation}
	
	$\bullet$ $L^{\infty} _{\mathcal{F}_{N-1} }(\Omega ;\mathbb H)$: the set of all $\mathbb H$-valued
	$\mathcal{F}_{N-1}$-measurable essentially bounded variables.

	$\bullet$ $L_{\mathbb{F}}^2\left(\mathbb{T}; \mathbb{H}\right)$: the set of all
	$\mathbb{H}$-valued  stochastic process $f(\cdot)=\{f_k|f_k$ is $\mathcal{F}_{k-1}$-measurable, k$\in \mathbb{T}\}$ satisfying
	
	\begin{equation}
		\|f(\cdot)\|_{L_{\mathbb{F}}^2\left(\mathbb{T}; \mathbb{H}\right)} :=\bigg [\mathbb{E}\bigg (\sum_{k=0}^{N -1} \|f_k\|_{\mathbb{H}}^{2}\bigg ) \bigg]^{\frac{1}{2}}<\infty.\nonumber
	\end{equation}
	
	$\bullet$ $L_{\mathbb{F}}^{\infty}(\mathbb{T}; \mathbb{H})$: the set of all
	$\mathbb{H}$-valued  essentially bounded stochastic processes $f(\cdot)=\{f_k|f_k$ is $\mathcal{F}_{k-1}$-measurable, k$\in \mathbb{T}\}$.
	
	$\bullet$ $\mathbb{U}$ : the  set of all admissible control
	$$
	\mathbb{U}:=\left\{\mathbf{u}=\left(u_0, u_1, \ldots, u_{N-1}\right) \mid u_k \text { is } \mathcal{F}_{k-1} \text {-measurable, } u_k \in U_k, \mathbb{E}\left[\sum_{k=0}^{N-1}\left|u_k\right|^2\right]<\infty\right\},
	$$
	where $	\{U_k\}_{k \in \mathbb{T}} $ is a sequence of nonempty convex subset of $\mathbb{R}^m$.
	
For the sake of simplicity of notation, we will also present some product space as follows:

$\bullet$ $N^{2}(\overline{\mathbb{T}};\mathbb{R}^{2n} ):= L_{\mathbb{F}}^2\left(\overline{\mathbb{T}}; \mathbb{R}^n\right)\times L_{\mathbb{F}}^2\left(\overline{\mathbb{T}}; \mathbb{R}^n\right)$ . For any $(x(\cdot ),y(\cdot ) ) \in N^{2}(\overline{\mathbb{T}} ;\mathbb{R}^{2n} )$, its norm is given by
\begin{equation}
	\begin{aligned}
		\|(x(\cdot),y(\cdot)) \|_{N^{2}(\overline{\mathbb{T}};\mathbb{R}^{2n} )}:= \left \{ \mathbb{E}\bigg [\displaystyle\sum_{k=0}^N|x(k)|^{2}+\displaystyle\sum_{k=0}^N|y(k)|^{2} \bigg ] \right \}^{\frac{1}{2}}.\nonumber
	\end{aligned}
\end{equation}

In this section, we introduce the controlled  fully coupled forward-backward stochastic difference equation (FBS$\Delta$E) system considered in this paper, along with the associated optimal control problem. This system is a discrete-time counterpart of continuous-time forward-backward stochastic differential equations (FBSDEs) and is widely used in stochastic control, mathematical finance, and other fields involving dynamic decision-making under uncertainty. Below, we provide a detailed description of the system and its components.   For  any admissible control \( u(\cdot) \in \mathbb{U} \), the system is described by the following equations:
\begin{equation}  \label{eq:2.1}
\left\{\begin{aligned}
&x_{k+1}  =b\left(k, x_k, y'_{k+1}, z'_{k+1}, u_k\right) + \sigma(k, x_k, y'_{k+1}, z'_{k+1}, u_k) \omega_k, \\
&y_k = -f\left(k+1, x_k, y'_{k+1}, z'_{k+1}, u_k\right), \\
&y'_{k+1} = \mathbb{E}[y_{k+1} | \mathcal{F}_{k-1}], \\
&z'_{k+1} = \mathbb{E}[y_{k+1} w_k | \mathcal{F}_{k-1}], \\
&x_0 = \Lambda(y_0), \\
&y_N = \Phi(x_N), \quad k \in \mathbb{T},
\end{aligned}\right.
\end{equation}
The system is fully coupled, meaning that the forward and backward processes \( x(\cdot) \) and \( y(\cdot) \) are interdependent. Specifically:The forward process \( x_{k+1} \) depends on the current state \( x_k \), the conditional expectations \( y'_{k+1} \) and \( z'_{k+1} \), and the control \( u_k \). The backward process \( y_k \) is determined by the forward state \( x_k \), the conditional expectations \( y'_{k+1} \) and \( z'_{k+1} \), and the control \( u_k \).

The performance of any admissible control \( u(\cdot) \in \mathbb{U} \) is evaluated using the following cost functional:

\begin{equation}  \label{eq:2.2}
J(u(\cdot)) = \mathbb{E}\left\{\varphi(x_N) + \gamma(y_0) + \sum_{k=0}^{N-1} l(k, x_k, y'_{k+1}, z'_{k+1}, u_k)\right\},
\end{equation}
where
 \( \varphi(x_N) \) is the terminal cost, depending on the final state \( x_N \).
 \( \gamma(y_0) \) is the initial cost, depending on the initial backward state \( y_0 \).
 \( l(k, x_k, y'_{k+1}, z'_{k+1}, u_k) \) is the running cost, depending on the current state \( x_k \), the conditional expectations \( y'_{k+1} \) and \( z'_{k+1} \), and the control \( u_k \).   Here $b, \sigma, f, l$ are given mappings such that: $ b: \Omega \times \mathbb{T}  \times \mathbb{R}^n \times \mathbb{R}^n\times \mathbb{R}^{n} \times \mathbb{R}^m\to \mathbb{R}^n,$
$ \sigma: \Omega \times \mathbb{T}  \times \mathbb{R}^n \times \mathbb{ R}^n\times \mathbb{R}^{n} \times \mathbb{R}^m\to \mathbb{R}^n,$
$f : \Omega \times \mathbb{T}  \times \mathbb{R}^n \times \mathbb{R}^{n} \times \mathbb{R}^n\times \mathbb{R}^m\to \mathbb{R}^n.$ For any $x ,y\in \mathbb{R}^n$, $\Phi (x),\varphi (x) \in L^{2} _{\mathcal{F}_{N-1} }(\Omega ; \mathbb {R}^n)$
and $\Lambda (y),\gamma(y)$ is deterministic.

The goal of the optimal control problem is to minimize the cost functional \( J(u(\cdot)) \) over all admissible control processes \( u(\cdot) \in \mathbb{U} \). Specifically, the optimal control problem for the fully coupled FBS$\Delta$E system is formulated as follows:

\begin{pro}\label{pro:2.1}
Find an admissible control \( u^*(\cdot) \in \mathbb{U} \) such that
\[
J(u^*(\cdot)) = \inf_{u(\cdot) \in \mathbb{U}} J(u(\cdot)).
\]
\end{pro}

Here, \( u^*(\cdot) \) is called the optimal control process, and the corresponding state process \( (x^*(\cdot), y^*(\cdot)) \) is referred to as the optimal state process. Together, the triplet \( (x^*(\cdot), y^*(\cdot), u^*(\cdot)) \) constitutes the optimal solution of the control problem.

To ensure that the optimal control problem is well-defined, we introduce the following assumptions. These assumptions provide the necessary regularity and structural conditions for the existence and uniqueness of solutions to the fully coupled FBS$\Delta$E system, as well as the well-posedness of the optimal control problem.

\begin{ass}\label{ass:2.3}
(i) For any \( x, y, z \in \mathbb{R}^n \), \( u \in U_k \), \( b(k, x, y, z, u) \) and \( \sigma(k, x, y, z, u) \) are \( \mathcal{F}_{k-1} \)-measurable, \( k \in \mathbb{T} \), and \( b(\cdot, 0, 0, 0, 0) \in L_{\mathbb{F}}^2(\mathbb{T}; \mathbb{R}^n) \), \( \sigma(\cdot, 0, 0, 0, 0) \in L_{\mathbb{F}}^2(\mathbb{T}; \mathbb{R}^n) \).

(ii) For any \( x, y', z' \in \mathbb{R}^n \), \( u \in U_k \), \( f(k+1, x, y', z', u) \) is \( \mathcal{F}_{k-1} \)-measurable, \( k \in \mathbb{T} \), and \( f(\cdot, 0, 0, 0, 0) \in L_{\mathbb{F}}^2(\mathbb{T}; \mathbb{R}^n) \).

(iii) For any \( (\omega, k) \in \Omega \times \mathbb{T} \), \( b, \sigma \), and \( f \) are differentiable with respect to \( (x, y, z, u) \). The corresponding derivatives \( \phi_a \), where \( \phi := b, \sigma, f \) and \( a := x, y, z, u \), are continuous and uniformly bounded.

(iv) For any \( (\omega, k) \in \Omega \times \mathbb{T} \), \( l \) is differentiable with respect to \( (x, y, z, u) \) with continuous derivatives \( l_a \), where \( a := x, y, z, u \). For any \( \omega \in \Omega \), \( g_1 \) and \( g_2 \) are differentiable with respect to \( x \) and \( y \) with continuous derivatives \( g_x \) and \( g_y \), where \( g_1 := \Phi, \varphi \), and \( g_2 := \Lambda, \gamma \). Moreover, for any \( (\omega, k) \in \Omega \times \mathbb{T} \), there exists a constant \( C > 0 \) such that for all \( (x, y, z, u) \in \mathbb{R}^n \times \mathbb{R}^n \times \mathbb{R}^n \times U_k \) and \( k \in \mathbb{T} \),
\[
|l(k, x, y, z, u)| \leq C(1 + |x|^2 + |y|^2 + |z|^2 + |u|^2),
\]
and
\[
|l_a(k, x, y, z, u)| \leq C(1 + |x| + |y| + |z| + |u|),
\]
where \( a := x, y, z, u \). Additionally, for any \( (\omega, k) \in \Omega \times \mathbb{T} \), there exists a constant \( C > 0 \) such that for any \( x, y \in \mathbb{R}^n \),
\[
|g_{1x}(x)| \leq C(1 + |x|), \quad |g_{2y}(y)| \leq C(1 + |y|),
\]
where \( g_1 := \Phi, \varphi \), and \( g_2 := \Lambda, \gamma \).
\end{ass}

\begin{ass}\label{ass:2.2}
There exist two constants \( \mu \ge 0 \), \( v \ge 0 \), a matrix-valued random variable \( G \in L_{\mathcal{F}_{N-1}}(\Omega; \mathbb{R}^{\tilde{m} \times n}) \), and a series of matrix-valued processes \( A(\cdot), B(\cdot), C(\cdot) \in L_{\mathbb{F}}^{\infty}(\mathbb{T}; \mathbb{R}^{m \times n}) \), such that the following conditions hold:

(i) One of the following two cases holds true:
 Case 1: \( \mu > 0 \) and \( v = 0 \);
 Case 2: \( \mu = 0 \) and \( v > 0 \).

(ii) (Domination condition) For all \( k \in \mathbb{T} \), almost all \( \omega \in \Omega \), \( u \in \mathbb{U} \), and any \( x, \bar{x}, y, \bar{y}, z, \bar{z} \in \mathbb{R}^n \),
\begin{equation}
\left\{\begin{aligned}
&|\Lambda(y) - \Lambda(\bar{y})| \le \frac{1}{\mu}|M\widehat{y}|, \\
&|\Phi(x) - \Phi(\bar{x})| \le \frac{1}{v}|G\widehat{x}|, \\
&|f(k+1, x, y, z, u) - f(k+1, \bar{x}, y, z, u)| \le \frac{1}{v}|A_k\widehat{x}|, \\
&|h(k, x, y, z, u) - h(k, x, \bar{y}, \bar{z}, u)| \le \frac{1}{\mu}|B_k\widehat{y} + C_k\widehat{z}|,
\end{aligned}\right.
\end{equation}
where \( h = b, \sigma \), and \( \widehat{x} = x - \bar{x} \), \( \widehat{y} = y - \bar{y} \), \( \widehat{z} = z - \bar{z} \).

(iii) (Monotonicity condition) For all \( k \in \mathbb{T} \), almost all \( \omega \in \Omega \), and any \( \theta = (x, y, z), \bar{\theta} = (\bar{x}, \bar{y}, \bar{z}) \in \mathbb{R}^{3n} \), one of the following two cases holds true:
\begin{equation}
\left\{\begin{aligned}
&\langle \Lambda(y) - \Lambda(\bar{y}), \widehat{y} \rangle \le -\mu|M\widehat{y}|^2, \\
&\langle \Phi(x) - \Phi(\bar{x}), \widehat{x} \rangle \ge v|G\widehat{x}|^2, \\
&\langle \Gamma(k, \theta, u) - \Gamma(k, \bar{\theta}, u), \widehat{\theta} \rangle \le -v|A_k\widehat{x}|^2 - \mu|B_k\widehat{y} + C_k\widehat{z}|^2,
\end{aligned}\right.
\end{equation}
or
\begin{equation}
\left\{\begin{aligned}
&\langle \Lambda(y) - \Lambda(\bar{y}), \widehat{y} \rangle \ge \mu|M\widehat{y}|^2, \\
&\langle \Phi(x) - \Phi(\bar{x}), \widehat{x} \rangle \le -v|G\widehat{x}|^2, \\
&\langle \Gamma(k, \theta, u) - \Gamma(k, \bar{\theta}, u), \widehat{\theta} \rangle \ge v|A_k\widehat{x}|^2 + \mu|B_k\widehat{y} + C_k\widehat{z}|^2,
\end{aligned}\right.
\end{equation}
where \( \Gamma(k, \theta, u) := (f(k+1, \theta, u), b(k, \theta, u), \sigma(k, \theta, u)) \), \( \langle \Gamma(k, \theta, u), \theta \rangle := \langle f(k+1, \theta, u), x \rangle + \langle b(k, \theta, u), y \rangle + \langle \sigma(k, \theta, u), z \rangle \), and \( \widehat{\theta} := \theta - \bar{\theta} \).
\end{ass}

Under the assumptions introduced above, we have the following key result regarding the existence, uniqueness, and a priori estimates for solutions to the fully coupled FBS$\Delta$E system. This result is based on Theorem 3.1 in \cite{ref-34}, which provides a general framework for analyzing such systems.

\begin{lem}\label{thm:2.3}
Let \( (\Lambda, \Phi, \Gamma) \) be a set of coefficients satisfying Assumptions \ref{ass:2.3} and \ref{ass:2.2}. Then, for any \( u(\cdot) \in \mathbb{U} \), the FBS$\Delta$E \eqref{eq:2.1} admits a unique solution \( (x(\cdot), y(\cdot)) \in N^{2}(\overline{\mathbb{T}}; \mathbb{R}^{2n}) \). Moreover, the following estimate holds:
\begin{equation}\label{eq:2.10}
\mathbb{E}\left[\sum_{k=0}^N |x_k|^2 + \sum_{k=0}^N |y_k|^2\right] \le K \mathbb{E}[\mathrm{I}],
\end{equation}
where
\[
\mathrm{I} = |\Phi(0)|^2 + \sum_{k=0}^{N-1} |b(k, 0, 0, 0, u_k)|^2 + \sum_{k=0}^{N-1} |\sigma(k, 0, 0, 0, u_k)|^2 + \sum_{k=0}^{N-1} |f(k+1, 0, 0, 0, u_k)|^2 + |\Lambda(0)|^2,
\]
and \( K \) is a positive constant depending only on \( \mathbb{T} \), the Lipschitz constants, \( \mu \), \( v \), and the bounds of \( G \), \( A(\cdot) \), \( B(\cdot) \), and \( C(\cdot) \).

Furthermore, let \( \bar{u}(\cdot) \in \mathbb{U} \) and \( (\bar{x}(\cdot), \bar{y}(\cdot)) \in N^{2}(\overline{\mathbb{T}}; \mathbb{R}^{2n}) \) be the corresponding solution to FBS$\Delta$E \eqref{eq:2.1}. Then the following estimate holds:
\begin{equation}\label{eq:2.8}
\mathbb{E}\left[\sum_{k=0}^N |\widehat{x}_k|^2 + \sum_{k=0}^N |\widehat{y}_k|^2\right] \le K \mathbb{E}[\widehat{\mathrm{I}}],
\end{equation}
where we denote :
	\begin{eqnarray}  \label{eq:2.9}
		\begin{aligned}	
			\widehat{\mathrm{I}}=&\sum_{k=0}^{N-1}|b(k,\bar{\theta}_k,u_k)-{b}(k,\bar{\theta}_k, \bar u_k)|^2+\sum_{k=0}^{N-1}|\sigma(k,\bar{\theta}_k, u_k)-{\sigma}(k,\bar{\theta}_k,\bar u_k)|^2
			\\&+\sum_{k=0}^{N-1}|f(k+1,\theta_k, u_k)-{f}(k+1,\bar{\theta}_k,\bar u_k)|^2,
		\end{aligned}
	\end{eqnarray}
		where $\theta_k =(x_k^\top ,y_{k+1 }'^\top,z_{k+1 }'^\top)^\top=(x_k^\top ,\mathbb{E} [y_{k+1 }|\mathcal{F}_{k-1}]^\top,\mathbb{E} [y_{k+1 }w_k|\mathcal{F}_{k-1}]^\top)^\top$, $
		\bar{\theta}_k =(\bar{x}_k^\top ,\bar{y}_{k+1 }'^\top,\bar{z}_{k+1 }'^\top)^\top$
		and $K$ is the same constant as in \eqref{eq:2.10}.
\end{lem}

The a priori estimates provided by Lemma \ref{thm:2.3} ensure that the performance index \( J(u(\cdot)) \) is well-defined. Specifically: The estimates \eqref{eq:2.10} and \eqref{eq:2.8} guarantee that the state processes \( x(\cdot) \) and \( y(\cdot) \) are square-integrable, which is necessary for the cost functional \( J(u(\cdot)) \) to be finite.  The stability estimate \eqref{eq:2.8} ensures that small perturbations in the control process lead to small changes in the state processes, which is essential for the continuity of the performance index with respect to the control. Therefore, the optimal control problem is well-posed.

In this paper, our goal is to establish  necessary and sufficient conditions for the stochastic maximum principle in the context of the optimal control problem for fully coupled forward-backward stochastic difference equations (FBS$\Delta$Es). The key idea is to use the Hamiltonian system to represent the variations of the cost
functional. To this end, we first introduce the adjoint equation associated with the state equation \eqref{eq:2.1}, which plays a central role in deriving the stochastic maximum principle. The adjoint equation, along with the Hamiltonian function, provides the necessary tools for analyzing the sensitivity of the cost functional to changes in the state and control processes. Below, we present the adjoint equation, define the Hamiltonian, and express the adjoint equation in terms of the Hamiltonian.

For ease of representation, we define the following shorthand notations for the coefficients of the state equation:
\[
\left\{\begin{array}{l}
b(k) := b\left(k, x_k, y_{k+1}^{\prime}, z_{k+1}^{\prime}, u_k\right), \\
\sigma(k) := \sigma\left(k, x_k, y_{k+1}^{\prime}, z_{k+1}^{\prime}, u_k\right), \\
f(k+1) := f\left(k+1, x_k, y_{k+1}^{\prime}, z_{k+1}^{\prime}, u_k\right), \\
l(k) := l\left(k, x_k, y_{k+1}^{\prime}, z_{k+1}^{\prime}, u_k\right).
\end{array}\right.
\]
The adjoint equation associated with the state equation \eqref{eq:2.1} is given by:
\begin{equation} \label{eq:2.16}
\left\{\begin{aligned}
r_{k+1} = & -\left[b_y^{\top}(k) p'_{k+1} + \sigma_y^{\top}(k) q'_{k+1} + f_y^{\top}(k+1) r_k + l_y(k)\right] \\
& -\left[b_z^{\top}(k) p'_{k+1} + \sigma_z^{\top}(k) q'_{k+1} + f_z^{\top}(k+1) r_k + l_z(k)\right] \omega_k, \\
p_k = & \quad b_x^{\top}(k) p'_{k+1} + \sigma_x^{\top}(k) q'_{k+1} + f_x^{\top}(k+1) r_k + l_x(k), \\
r_0 = & -\gamma_y\left(y_0\right) - \Lambda_y^{\top}\left(y_0\right) \cdot p_0, \\
p_N = & \quad \varphi_x\left(x_N\right) - \Phi_x^{\top}\left(x_N\right) \cdot r_N,
\end{aligned}\right.
\end{equation}
where:
 \( q_k := p_k \omega_{k-1} \),
 \( p'_{k+1} = \mathbb{E}[p_{k+1} | \mathcal{F}_{k-1}] \),
 \( q'_{k+1} = \mathbb{E}[p_{k+1} \omega_k | \mathcal{F}_{k-1}] \).

 The adjoint equation \eqref{eq:2.16} is a fully coupled backward stochastic difference equation (BS$\Delta$E), characterized by the mutual dependence between the adjoint processes \( (p_k, r_k) \). This interdependence means that the evolution of \( p_k \) is influenced by \( r_k \), and vice versa, reflecting the intrinsic coupling of the system. The solution to this equation is a pair of processes \( (p_k, r_k) \), which together encode the sensitivity of the cost functional to variations in the state processes \( (x_k, y_k) \). As a key component of the stochastic maximum principle, the adjoint equation \eqref{eq:2.16} establishes the necessary conditions for optimality by linking the state processes \( (x_k, y_k) \) with the adjoint processes \( (p_k, r_k) \). Specifically, the adjoint processes act as Lagrange multipliers, capturing the sensitivity of the cost functional to changes in the state variables. These processes are instrumental in deriving the optimal control through the Hamiltonian, which provides a unified framework for analyzing the interaction between the state, adjoint, and control processes.

To further analyze the optimal control problem, we introduce the Hamiltonian function \( H \), defined as:
\[
H: \Omega \times \mathbb{T} \times \mathbb{R}^n \times \mathbb{R}^n \times \mathbb{R}^n \times \mathbb{R}^m \times \mathbb{R}^n \times \mathbb{R}^n \times \mathbb{R}^n \to \mathbb{R},
\]
and given by:
\[
H(k, x, y, z, u, p, q, r) = b^\top(k, x, y, z, u) p + \sigma^\top(k, x, y, z, u) q + f^\top(k+1, x, y, z, u) r + l(k, x, y, z, u).
\]
Using the Hamiltonian, the adjoint equation \eqref{eq:2.16} can be expressed in a more compact form as:
\begin{equation} \label{eq:2.17}
\left\{\begin{aligned}
r_{k+1} = & -H_y(k) - H_z(k) \omega_k, \\
p_k = & \quad H_x(k), \\
r_0 = & -\gamma_y\left(y_0\right) - \Lambda_y^{\top}\left(y_0\right) \cdot p_0, \\
p_N = & \quad \varphi_x\left(x_N\right) - \Phi_x^{\top}\left(x_N\right) \cdot r_N,
\end{aligned}\right.
\end{equation}
where the partial derivatives of the Hamiltonian are defined as:
\[
\left\{
\begin{aligned}
H_x(k) &:= \frac{\partial H}{\partial x}(k, x_k, y_{k+1}^{\prime}, z_{k+1}^{\prime}, u_k, p_{k+1}^{\prime}, q_{k+1}^{\prime}, r_k), \\
H_y(k) &:= \frac{\partial H}{\partial y^{\prime}}(k, x_k, y_{k+1}^{\prime}, z_{k+1}^{\prime}, u_k, p_{k+1}^{\prime}, q_{k+1}^{\prime}, r_k), \\
H_z(k) &:= \frac{\partial H}{\partial z^{\prime}}(k, x_k, y_{k+1}^{\prime}, z_{k+1}^{\prime}, u_k, p_{k+1}^{\prime}, q_{k+1}^{\prime}, r_k), \\
H_u(k) &:= \frac{\partial H}{\partial u}(k, x_k, y_{k+1}^{\prime}, z_{k+1}^{\prime}, u_k, p_{k+1}^{\prime}, q_{k+1}^{\prime}, r_k).
\end{aligned}
\right.
\]

In the next section, we will use this coupled adjoint equation, along with the Hamiltonian, to derive the necessary and sufficient conditions for optimality, leading to the stochastic maximum principle for the fully coupled FBS$\Delta$E system.

	\section{Stochastic maximum principle for the coupled FBS$\bigtriangleup $E system} \label{sec:3}

In order to derive the stochastic maximum principle for Problem~\ref{pro:2.1}, we adopt a variational analysis framework based on spike perturbation techniques. The central idea is to perturb the control at a single time point and analyze the effect on the coupled forward-backward stochastic difference equation (FBS$\bigtriangleup$E) system and the associated cost functional.

Suppose that \( u^*(\cdot) \) is the optimal control of Problem \ref{pro:2.1}, and \( (x^*(\cdot), y^*(\cdot)) \) is the corresponding optimal trajectory. For a fixed time $ s \in \mathbb{T}$, choose any control \( u(\cdot) \in \mathbb{U} \) such that \(   u(\cdot)-u^*(\cdot) \in \mathbb{U} \). For any \( \varepsilon \in [0, 1] \), we construct the perturbed admissible control as follows:
\[
u^{\varepsilon}_k = (1 - \delta_{ks}) u^*_k + \delta_{ks} (u^*_k + \varepsilon (u_k-u^*_k)) = u^*_k + \delta_{ks} \varepsilon (u_k-u^*_k),
\qquad \qquad 0 \leq \varepsilon \leq 1, \, k \in \mathbb{T},
\]
where \( \delta_{ks} \) is the Kronecker delta function, defined by:
\[
\delta_{ks} =
\begin{cases}
1, & \text{if } k = s, \\
0, & \text{if } k \neq s.
\end{cases}
\]

Since \( \mathbb{U} \) is a convex set, the perturbed control \( \{u^{\varepsilon}_k\}_{k=0}^N \) remains admissible, i.e., \( u^{\varepsilon}(\cdot) \in \mathbb{U} \). Let \( (x^{\varepsilon}(\cdot), y^{\varepsilon}(\cdot)) \) be the solution of \eqref{eq:2.1} corresponding to the perturbed control \( u^{\varepsilon}(\cdot) \), and let \( (x^*(\cdot), y^*(\cdot)) \) be the solution of \eqref{eq:2.1} corresponding to the optimal control \( u^*(\cdot) \).

In what follows, we first present the estimates of    the perturbed state process \((x^{\varepsilon}(\cdot), y^{\varepsilon}(\cdot)) \).

\begin{lem}  \label{lem:3.1}
		Suppose that the conditions of Assumptions \ref{ass:2.3}  and \ref{ass:2.2}  are satisfied. Then \eqref{eq:2.1} yields the following estimate:

			\begin{equation}   \label{eq:3.2}
				\mathbb{E}\left[\sum_{k=0}^N\left|x_k^{\varepsilon}-x_k^*\right|^2\right]+\mathbb{E}\bigg[\displaystyle\sum_{k=0}^N |y^\varepsilon_{k}-y^*_{k}|^2\bigg]=O\left(\varepsilon^2\right)
			\end{equation}

\end{lem}
	
	\begin{proof}
By the estimate  \eqref{eq:2.9}	and the boundedness property of Fréchet derivatives $b_u$ , $\sigma_u$ and $f_u$, we have
		\begin{equation}
			\begin{aligned}
				\mathbb {E} \bigg[\sum_{k=0}^N\left|x_k^{\varepsilon}-x_k^*\right|^2 \bigg]+\mathbb{E}\bigg[\displaystyle\sum_{k=0}^N |y^\varepsilon_{k}-y^*_{k}|^2\bigg]
				 & \leqslant C \mathbb {E}\left[\sum_{k=0}^{N-1}|u_k^{\varepsilon}-u_k^*|^2\right] \\
				 &=C \mathbb {E}\left[\sum_{k=0}^{N-1}| u^*_k+\delta_{k s} \varepsilon (u_k-u^*_k)-u_k^*|^2\right] \\
				& =C \varepsilon^2  \mathbb {E}\bigg[|u_s-u^*_s|^2\bigg] \\
				& =O\left(\varepsilon^2\right). \nonumber
			\end{aligned}
		\end{equation}

	\end{proof}
	
	Next we represent the difference $J\left(u^\varepsilon(\cdot)\right)-J\left(u^*(\cdot)\right)$ in terms of the Hamiltonian $H$ and the state process $(x(\cdot), y(\cdot))$ as well as other relevant expressions.  To simplify our notation, we will also use the following abbreviations
	\begin{equation}
		\left\{\begin{aligned}
			\phi^\varepsilon(k)&:=\phi\left(k, x_k^\varepsilon, y_{k+1}^{\prime \varepsilon}, z_{k+1}^{\prime \varepsilon}, u_k^{\varepsilon}\right) ,\\
			\phi^*(k)&:=\phi\left(k, x_k^*, y_{k+1}^{\prime *}, z_{k+1}^{\prime *}, u_k^{*}\right) ,  \quad \phi := b, \sigma, f ,\\
			H^\varepsilon(k)&:=H\left(k, x_k^{\varepsilon}, y_{k+1}^{\prime \varepsilon}, z_{k+1}^{\prime \varepsilon}, u_k^\varepsilon, p_{k+1}^{\prime *}, q_{k+1}^{\prime *}, r_k^*\right),\\
			H^*(k)&:=H\left(k, x_k^*, y_{k+1}^{\prime *}, z_{k+1}^{\prime *}, u_k^*, p_{k+1}^{\prime *}, q_{k+1}^{\prime *}, r_k^*\right). \nonumber
		\end{aligned}\right.
	\end{equation}
	
	\begin{lem}  \label{lem:3.2}
		Suppose that the conditions of Assumptions \ref{ass:2.3} and \ref{ass:2.2} are satisfied.
		Then we have
			\begin{equation} \label{eq:3.4}
		\begin{aligned}
			&J\left(u^{\varepsilon}(\cdot)\right)-J\left(u^*(\cdot)\right)\\
			= \mathbb{E} \bigg\{
			&\sum_{k=0}^{N-1}\bigg[H^\varepsilon(k)-H^*(k)-\left\langle x_k^{\varepsilon}-x_k^*, H_x^*(k)\right\rangle
			-\left\langle y^{\prime \varepsilon} _{k+1}-y^{\prime *}_{k+1}, H_y^*(k)\right\rangle-\left\langle z_{k+1}^{\prime \varepsilon}-z_{k+1}^{\prime *}, H_z^*(k)\right\rangle\bigg]\\
			&+\varphi^{\varepsilon}\left(x_N^{\varepsilon}\right)-\varphi^*\left(x_N^*\right)-\langle x_N^\varepsilon-x_N^*,\varphi^*_x(x^*_N)  \rangle +\gamma^{\varepsilon}\left(y_0^\varepsilon\right)-\gamma^*\left(y_0^*\right)-\left\langle y_0^{\varepsilon}-y_0^*, \gamma^*_y\left(y_0^*\right)\right\rangle \\
			&+\left\langle\Phi_x^*\left(x_N^*\right)\left(x_N^{\varepsilon}-x_N^*\right)-\left(\Phi^{\varepsilon}\left(x_N^\varepsilon\right)-\Phi^*\left(x_N^*\right)\right),r_N^*\right\rangle
			+\left\langle\Lambda^{\varepsilon}\left(y_0^\varepsilon\right)-\Lambda^*\left(y_0^*\right)-\Lambda_y^*\left(y_0^*\right)\left(y_0^{\varepsilon}-y_0^*\right), p_0^*\right\rangle
			\bigg\}
		\end{aligned}
	\end{equation}
	\end{lem}

{
\begin{proof}
By direct substitution of the definitions of the cost functional $J(u(\cdot))$ (Eq. \eqref{eq:2.2}) and the Hamiltonian $H$, we derive the initial decomposition of the cost difference:
\begin{equation} \label{eq:3.5}
\begin{aligned}
J\left(u^{\varepsilon}(\cdot)\right)-J\left(u^*(\cdot)\right) = \mathbb{E} \bigg\{
&\sum_{k=0}^{N-1}\bigg[H^\varepsilon(k)-H^*(k)-\left\langle b^\varepsilon(k)-b^*(k), p_{k+1}^{\prime *}\right\rangle-\left\langle\sigma^\varepsilon(k)-\sigma^*(k), q_{k+1}^{\prime *}\right\rangle\\
&-\left\langle f^{\varepsilon}{(k+1)}-f^*(k+1), r_k^*\right\rangle\bigg]
+\varphi^{\varepsilon}\left(x_N^{\varepsilon}\right)-\varphi^*\left(x_N^*\right)+\gamma^{\varepsilon}\left(y_0^{\varepsilon}\right)-\gamma^*\left(y_0^*\right)
 \bigg\}.
\end{aligned}
\end{equation}

 Step 1: Telescoping sum for adjoint process $p_k.$
We first leverage the "telescoping sum property" for the adjoint process $p_k$ associated with the forward state $x_k$:
\begin{equation} \label{eq:3.6}
\begin{aligned}
& \sum_{k=0}^{N-1}\left[\left\langle x_{k+1}^{\varepsilon}-x_{k+1}^*, p_{k+1}^*\right\rangle-\left\langle x_k^{\varepsilon}-x_k^*, p_k^*\right\rangle\right] \\
= & \left\langle x_N^{\varepsilon}-x_N^*, p_N^*\right\rangle-\left\langle x_0^{\varepsilon}-x_0^*, p_0^*\right\rangle.
\end{aligned}
\end{equation}
Substitute the forward state dynamics $x_{k+1} = b(k) + \sigma(k)\omega_k$ into Eq. \eqref{eq:3.6} and use the definition of the Hamiltonian partial derivative $H_x^*(k) = p_k^*$ (from the adjoint equation):
\begin{equation} \label{eq:3.7}
\begin{aligned}
& \left\langle x_N^{\varepsilon}-x_N^*, p_N^*\right\rangle-\left\langle x_0^{\varepsilon}-x_0^*, p_0^*\right\rangle \\
= & \sum_{k=0}^{N-1}\left[\left\langle b^{\varepsilon}(k)-b^*(k)
+(\sigma^{\varepsilon}(k)-\sigma^*(k)) \omega_k, p_{k+1}^*\right\rangle-\left\langle x_k^{\varepsilon}-x_k^*, H_x^*(k)\right\rangle\right] \\
= & \sum_{k=0}^{N-1}\left[\left\langle b^{\varepsilon}(k)-b^*(k), p_{k+1}^*\right\rangle+\left\langle\sigma^{\varepsilon}(k)-\sigma^*(k), q_{k+1}^*\right\rangle-\left\langle x_k^{\varepsilon}-x_k^*, H_x^*(k)\right\rangle\right],
\end{aligned}
\end{equation}
where $q_{k+1}^* = \mathbb{E}[p_{k+1}^*\omega_k \mid \mathcal{F}_k]$ (conditional covariance of $p_{k+1}^*$ and $\omega_k$). Rearranging Eq. \eqref{eq:3.7} yields:
\begin{equation} \label{eq:3.8}
\begin{aligned}
& \sum_{k=0}^{N-1}\left[-\left\langle b^{\varepsilon}(k)-b^*(k), p_{k+1}^*\right\rangle -\left\langle\sigma^{\varepsilon}(k)-\sigma^*(k), q_{k+1}^*\right\rangle\right] \\
= & \sum_{k=0}^{N-1}\bigg[-\left\langle x_k^{\varepsilon}-x_k^*, H_x^*(k)\right\rangle\bigg]-\left\langle x_N^\varepsilon-x_N^*, p_N^*\right\rangle+\left\langle x_0^\varepsilon-x_0^*, p_0^*\right\rangle.
\end{aligned}
\end{equation}

 Step 2: Telescoping sum for adjoint process $r_k.$

We next analyze the telescoping sum for the adjoint process $r_k$ associated with the backward state $y_k$ (correcting the index error in the original sum):
\begin{equation} \label{eq:3.9}
\begin{aligned}
& \sum_{k=0}^{N-1}\left[\left\langle y_{k+1}^\varepsilon-y_{k+1}^*, r_{k+1}^*\right\rangle-\left\langle y_k^{\varepsilon}-y_k^*, r_k^*\right\rangle\right] \\
=& \left\langle y_N^{\varepsilon}-y_N^*, r_N^*\right\rangle-\left\langle y_0^{\varepsilon}-y_0^*, r_0^*\right\rangle.
\end{aligned}
\end{equation}
Substitute the backward state dynamics $y_k = -H_y^*(k) - H_z^*(k)\omega_k + f(k+1)$ into Eq. \eqref{eq:3.9}:
\begin{equation} \label{eq:3.10}
\begin{aligned}
& \left\langle y_N^{\varepsilon}-y_N^*, r_N^*\right\rangle-\left\langle y_0^{\varepsilon}-y_0^*, r_0^*\right\rangle \\
= & \sum_{k=0}^{N-1}\left[\left\langle y_{k+1}^\varepsilon-y_{k+1}^*,-H_y^*(k)-H_z^* (k)\omega_k\right\rangle+\left\langle f^\varepsilon(k+1)-f^*(k+1), r_k^*\right\rangle\right] \\
= & \sum_{k=0}^{N-1}\left[\left\langle y_{k+1}^\varepsilon-y_{k+1}^*,-H_y^*(k)\right\rangle+\langle z_{k+1}^\varepsilon-z_{k+1}^* ,  -H_z^*(k)\rangle    +\left\langle f^\varepsilon(k+1)-f^*(k+1), r_k^*\right\rangle\right].
\end{aligned}
\end{equation}
Rearranging Eq. \eqref{eq:3.10} to isolate the $f$-term gives:
\begin{equation} \label{eq:3.11}
\begin{aligned}
&\sum_{k=0}^{N-1}\left[-\left\langle f^{\varepsilon}(k+1)-f^*(k+1), r_k^*\right\rangle\right] \\
=&\sum_{k=0}^{N-1}\left[-\left\langle y^{\varepsilon}_{k+1}-y^*_{k+1}, H_y^*(k)\right\rangle-\left\langle z_{k+1}^{\varepsilon}-z_{k+1}^*, H_z^*(k)\right\rangle\right]-\left\langle y_N^\varepsilon-y_N^*, r_N^*\right\rangle+\left\langle y_0^{\varepsilon}-y_0^*, r_0^*\right\rangle.
\end{aligned}
\end{equation}

 Step 3: Combine results and substitute boundary conditions.

Substitute Eqs. \eqref{eq:3.8} and \eqref{eq:3.11} into the initial cost difference decomposition (Eq. \eqref{eq:3.5}), and incorporate the boundary conditions from the adjoint system: Terminal condition: $p_N^* = \varphi_x^*(x_N^*) - \Phi_x^{* \top}(x_N^*) r_N^*$ (linking $p_N^*$ to the terminal cost $\varphi$ and terminal mapping $\Phi$); Initial condition: $r_0^* = -\gamma_y^*(y_0^*) - \Lambda_y^{* \top}(y_0^*) p_0^*$ (linking $r_0^*$ to the initial cost $\gamma$ and initial mapping $\Lambda$);
 Initial state constraint: $x_0^\varepsilon - x_0^* = \Lambda^\varepsilon(y_0^\varepsilon) - \Lambda^*(y_0^*)$;
Terminal state constraint: $y_N^\varepsilon - y_N^* = \Phi^\varepsilon(x_N^\varepsilon) - \Phi^*(x_N^*)$.

After substituting these boundary conditions and rearranging terms, we obtain:
\begin{align}
&J\left(u^{\varepsilon}(\cdot)\right)-J\left(u^*(\cdot)\right) \notag\\
= \mathbb{E} \bigg\{ &\sum_{k=0}^{N - 1}\bigg[H^\varepsilon(k)-H^*(k)-\left\langle x_k^{\varepsilon}-x_k^*, H_x^*(k)\right\rangle -\left\langle y^{\varepsilon}_{k + 1}-y^{*}_{k + 1}, H_y^*(k)\right\rangle-\left\langle z_{k + 1}^{\varepsilon}-z_{k + 1}^*, H_z^*(k)\right\rangle\bigg] \notag\\
&+\varphi^{\varepsilon}\left(x_N^{\varepsilon}\right)-\varphi^*\left(x_N^*\right)+\gamma^{\varepsilon}\left(y_0^{\varepsilon}\right)-\gamma^*\left(y_0^*\right)-\left\langle x_N^\varepsilon-x_N^*, p_N^*\right\rangle+\left\langle x_0^\varepsilon-x_0^*,p_0^*\right\rangle -\left\langle y_N^\varepsilon-y_N^*, r_N^*\right\rangle+\left\langle y_0^{\varepsilon}-y_0^*, r_0^*\right\rangle \bigg\} \notag\\
= \mathbb{E} \bigg\{ &\sum_{k=0}^{N - 1}\bigg[H^\varepsilon(k)-H^*(k)-\left\langle x_k^{\varepsilon}-x_k^*, H_x^*(k)\right\rangle -\left\langle y^{\varepsilon}_{k + 1}-y^{*}_{k + 1}, H_y^*(k)\right\rangle-\left\langle z_{k + 1}^{\varepsilon}-z_{k + 1}^*, H_z^*(k)\right\rangle\bigg] \notag\\
&+\varphi^{\varepsilon}\left(x_N^{\varepsilon}\right)-\varphi^*\left(x_N^*\right)+\gamma^{\varepsilon}\left(y_0^{\varepsilon}\right)-\gamma^*\left(y_0^*\right)-\left\langle x_N^\varepsilon-x_N^*,\varphi^*_x(x^*_N)-\Phi_x^{*\top}\left(x^*_N\right) \cdot r^*_N\right\rangle +\left\langle \Lambda^\varepsilon (y^\varepsilon_0)-\Lambda^* (y^*_0),p_0^*\right\rangle \notag\\
&-\left\langle \Phi^{\varepsilon}\left(x_N^{\varepsilon}\right)-\Phi^*\left(x_N^*\right), r_N^*\right\rangle+\left\langle y_0^{\varepsilon}-y_0^*, -\gamma^*_y\left(y^*_0\right)-\Lambda_y^{* \top}\left(y^*_0\right) \cdot p^*_0\right\rangle \bigg\} \notag\\
= \mathbb{E} \bigg\{ &\sum_{k=0}^{N - 1}\bigg[H^\varepsilon(k)-H^*(k)-\left\langle x_k^{\varepsilon}-x_k^*, H_x^*(k)\right\rangle -\left\langle y^{\varepsilon} _{k + 1}-y^{*}_{k + 1}, H_y^*(k)\right\rangle-\left\langle z^{\varepsilon} _{k + 1}-z^{*}_{k + 1}, H_z^*(k)\right\rangle\bigg] \notag\\
&+\varphi^{\varepsilon}\left(x_N^{\varepsilon}\right)-\varphi^*\left(x_N^*\right)-\langle x_N^\varepsilon-x_N^*,\varphi^*_x(x^*_N) \rangle +\gamma^{\varepsilon}\left(y_0^\varepsilon\right)-\gamma^*\left(y_0^*\right)-\left\langle y_0^{\varepsilon}-y_0^*, \gamma^*_y\left(y_0^*\right)\right\rangle \notag\\
&+\left\langle\Phi_x^*\left(x_N^*\right)\left(x_N^{\varepsilon}-x_N^*\right)-\left(\Phi^{\varepsilon}\left(x_N^\varepsilon\right)-\Phi^*\left(x_N^*\right)\right),r_N^*\right\rangle +\left\langle\Lambda^{\varepsilon}\left(y_0^\varepsilon\right)-\Lambda^*\left(y_0^*\right)-\Lambda_y^*\left(y_0^*\right)\left(y_0^{\varepsilon}-y_0^*\right), p_0^*\right\rangle \bigg\}.
\end{align}

 Step 4: Replace states with conditional expectations.

Finally, we replace the random variables $y^{\varepsilon}_{k+1}, y^*_{k+1}, z^{\varepsilon}_{k+1}, z^*_{k+1}$ with their $\mathcal{F}_k$-conditional expectations:
\[
y^{\prime \varepsilon}_{k + 1} = \mathbb{E}[y^{\varepsilon}_{k + 1} \mid \mathcal{F}_k], \quad y^{\prime *}_{k + 1} = \mathbb{E}[y^{*}_{k + 1} \mid \mathcal{F}_k],
\]
\[
z^{\prime \varepsilon}_{k + 1} = \mathbb{E}[z^{\varepsilon}_{k + 1} \mid \mathcal{F}_k], \quad z^{\prime *}_{k + 1} = \mathbb{E}[z^{*}_{k + 1} \mid \mathcal{F}_k].
\]
This substitution is valid due to the measurability and linearity of conditional expectation: the Hamiltonian $H(k)$ is $\mathcal{F}_k$-measurable, so inner products involving $\mathcal{F}_k$-measurable random variables preserve their values when replaced by conditional expectations over $\mathcal{F}_k$.
The combination of the above steps yields Eq. \eqref{eq:3.4}, completing the proof.
\end{proof}}

	Now we are in the position to use Lemmas \ref{lem:3.1} and \ref{lem:3.2}  to derive the variational formula for the cost functional $J(u(\cdot))$ in
	terms of the Hamiltonian $H$.

	\begin{thm}  \label{thm:3.3}
		Under Assumptions \ref{ass:2.3} and \ref{ass:2.2}, the first-order variational formula for the cost functional $J(u(\cdot))$ can be expressed as
	\begin{equation}
		\begin{aligned}
			 \frac{d}{d \varepsilon} J(u^\varepsilon(\cdot))\bigg|_{\varepsilon=0}  :&=\lim _{\varepsilon \rightarrow 0^{+}} \frac{J(u^\varepsilon(\cdot))-J(u^*(\cdot))}{\varepsilon} \\
			& =\mathbb{E}\bigg[\langle H_u^*(s), u_s-u_s^*\rangle \bigg]
		\end{aligned}
	\end{equation}
	
	\end{thm}
	
	\begin{proof}
		For notational simplicity, we write
		\begin{equation}
			\begin{aligned}
				\beta ^\varepsilon= \mathbb{E} \bigg\{
				&\sum_{k=0}^{N-1}\bigg[H^\varepsilon(k)-H^*(k)-\left\langle x_k^{\varepsilon}-x_k^*, H_x^*(k)\right\rangle
				-\left\langle y^{\prime \varepsilon}_{k+1}-y^{\prime *}_{k+1}, H_y^*(k)\right\rangle
-\left\langle z_{k+1}^{\prime \varepsilon}-z_{k+1}^{\prime *}, H_z^*(k)\right\rangle-\left\langle u_{k}^{\varepsilon}-u_{k}^*, H_u^*(k)\right\rangle\bigg]\\
				&+\varphi^{\varepsilon}\left(x_N^{\varepsilon}\right)-\varphi^*\left(x_N^*\right)-\langle x_N^\varepsilon-x_N^*,\varphi^*_x(x^*_N)  \rangle +\gamma^{\varepsilon}\left(y_0^\varepsilon\right)-\gamma^*\left(y_0^*\right)-\left\langle y_0^{\varepsilon}-y_0^*, \gamma^*_y\left(y_0^*\right)\right\rangle \\
				&+\left\langle\Phi_x^*\left(x_N^*\right)\left(x_N^{\varepsilon}-x_N^*\right)-\left(\Phi^{\varepsilon}\left(x_N^\varepsilon\right)-\Phi^*\left(x_N^*\right)\right),r_N^*\right\rangle
				+\left\langle\Lambda^{\varepsilon}\left(y_0^\varepsilon\right)-\Lambda^*\left(y_0^*\right)-\Lambda_y^*\left(y_0^*\right)\left(y_0^{\varepsilon}-y_0^*\right), p_0^*\right\rangle
				\bigg\}
			\end{aligned}
		\end{equation}
By Lemma \ref{lem:3.2}, we have
		
		\begin{equation} \label{eq:3.13}
						\begin{aligned}
		J\left(u^{\varepsilon}(\cdot)\right)-J\left(u^*(\cdot)\right) &=
		\beta ^\varepsilon  +\mathbb{E}\left[\sum_{k=0}^{N-1}\left\langle H_u^*(k), u_k^\eps-u_k^*\right\rangle \right]\\
		&=\beta ^\varepsilon + \varepsilon \mathbb{E}\bigg[\langle H_u^*(s), u_s-u_s^*\rangle \bigg].
						\end{aligned}
		\end{equation}
		Under Assumptions \ref{ass:2.3}, combining the Taylor Expansions we can get
		
\begin{equation}  \label{eq:3.131}
	\begin{aligned}
		H^{\varepsilon}(k)-H^*(k)=  \int_0^1\bigg\{ &
		\langle x_k^{\varepsilon}-x_k^*, H_x^{\varepsilon, \lambda} \rangle  +\langle y_{k+1}^{\prime \varepsilon}-y_{k+1}^{\prime *}, H_y^{\varepsilon, \lambda} \rangle\\
		&+\langle z_{k+1}^{\prime \varepsilon}-z_{k+1}^{\prime *}, H_z^{\varepsilon, \lambda} \rangle+\langle u_k^{\varepsilon}-u_k^*, H_u^{\varepsilon, \lambda} \rangle\bigg\} d \lambda
	\end{aligned}
\end{equation}
where
		$$
	\left\{	\begin{aligned}
			& H^{\varepsilon, \lambda}(k):=H\bigg(k, x_k^{\varepsilon, \lambda},
			y_{k+1}^{\prime \varepsilon, \lambda}, z_{k+1}^{\prime \varepsilon, \lambda}, u_k^{\varepsilon, \lambda}\bigg),\\
			&x_k^{\varepsilon, \lambda}=x_k^*+\lambda\bigg(x_k^{\varepsilon}-x_k^*\bigg),\\
			&u_k^{\varepsilon, \lambda}=u_k^*+\lambda\bigg( u_k^{\varepsilon}-u_k^*\bigg),\\
			&y_{k+1}^{\prime \varepsilon, \lambda} :=
			\mathbb{E}\bigg[    y_{k+1}^*+\lambda(y_{k+1}^{\varepsilon}-y_{k+1}^*)| \mathcal{F}_{k-1} \bigg],\\
			&z_{k+1}^{\prime \varepsilon, \lambda} :=
			\mathbb{E}\bigg[    z_{k+1}^*+\lambda(z_{k+1}^{\varepsilon}-z_{k+1}^*)| \mathcal{F}_{k-1} \bigg].
		\end{aligned} \right.
		$$	
{ By using\eqref{eq:3.131},  the Cauchy - Schwarz inequality, Lemma \ref{lem:3.1}, and the dominated convergence theorem, we can obtain the following estimate:
{\small
\begin{eqnarray}
\begin{split}
	\mathbb{E}\bigg\{\sum_{k = 0}^{N - 1}&\left[ H^\varepsilon(k)-H^*(k)-\left\langle x_k^{\varepsilon}-x_k^*, H_x^*(k)\right\rangle -\left\langle y^{\prime \varepsilon}_{k + 1}-y^{\prime *}_{k + 1}, H_y^*(k)\right\rangle \right.\\
	&\left. -\left\langle z_{k + 1}^{\prime \varepsilon}-z_{k + 1}^{\prime *}, H_z^*(k)\right\rangle -\left\langle u_{k}^{\varepsilon}-u_{k}^*, H_u^*(k)\right\rangle \right]\bigg\}\\
	={}&\mathbb{E}\bigg\{\sum_{k = 0}^{N - 1}\int_0^1\bigg\{ \langle x_k^{\varepsilon}-x_k^*, H_x^{\varepsilon, \lambda} - H_x^*(k)\rangle  +\langle y_{k+1}^{\prime \varepsilon}-y_{k+1}^{\prime *}, H_y^{\varepsilon, \lambda}- H_y^*(k)\rangle\\
	&+\langle z_{k+1}^{\prime \varepsilon}-z_{k+1}^{\prime *}, H_z^{\varepsilon, \lambda}- H_z^*(k)\rangle+\langle u_k^{\varepsilon}-u_k^*, H_u^{\varepsilon, \lambda}- H_u^*(k)\rangle\bigg\} d \lambda\bigg\}\\
	\leq{}&\sum_{k = 0}^{N - 1}\int_0^1\left[\sqrt{\mathbb{E}\left[|x_k^{\varepsilon}-x_k^*|^2\right]}\sqrt{\mathbb{E}\left[|H_x^{\varepsilon, \lambda} - H_x^*(k)|^2\right]} + \sqrt{\mathbb{E}\left[|y_{k+1}^{\prime \varepsilon}-y_{k+1}^{\prime *}|^2\right]}\sqrt{\mathbb{E}\left[|H_y^{\varepsilon, \lambda}- H_y^*(k)|^2\right]}\right.\\
	&\left. + \sqrt{\mathbb{E}\left[|z_{k+1}^{\prime \varepsilon}-z_{k+1}^{\prime *}|^2\right]}\sqrt{\mathbb{E}\left[|H_z^{\varepsilon, \lambda}- H_z^*(k)|^2\right]} + \sqrt{\mathbb{E}\left[|u_k^{\varepsilon}-u_k^*|^2\right]}\sqrt{\mathbb{E}\left[|H_u^{\varepsilon, \lambda}- H_u^*(k)|^2\right]}\right] d \lambda \\
	\leq{}&\sum_{k = 0}^{N - 1}\int_0^1\left[\sqrt{\mathbb{E}\left[|x_k^{\varepsilon}-x_k^*|^2\right]}\sqrt{\mathbb{E}\left[|H_x^{\varepsilon, \lambda} - H_x^*(k)|^2\right]} + \sqrt{\mathbb{E}\left[|y_{k + 1}^{\varepsilon}-y_{k + 1}^{*}|^2\right]}\sqrt{\mathbb{E}\left[|H_y^{\varepsilon, \lambda}- H_y^*(k)|^2\right]}\right.\\
	&\left. + \sqrt{\mathbb{E}\left[|y_{k + 1}^{\varepsilon}-y_{k + 1}^{*}|^2\right]}\sqrt{\mathbb{E}\left[|H_z^{\varepsilon, \lambda}- H_z^*(k)|^2\right]} + \sqrt{\mathbb{E}\left[|u_k^{\varepsilon}-u_k^*|^2\right]}\sqrt{\mathbb{E}\left[|H_u^{\varepsilon, \lambda}- H_u^*(k)|^2\right]}\right] d \lambda\\
	\leq{}&\sqrt{\sum_{k = 0}^{N - 1}\mathbb{E}\left[|x_k^{\varepsilon}-x_k^*|^2\right]}\int_0^1\sqrt{\sum_{k = 0}^{N - 1}\mathbb{E}\left[|H_x^{\varepsilon, \lambda} - H_x^*(k)|^2\right]}d\lambda \\
	&+ \sqrt{\sum_{k = 0}^{N - 1}\mathbb{E}\left[|y_{k + 1}^{\varepsilon}-y_{k + 1}^{*}|^2\right]}\int_0^1\left(\sqrt{\sum_{k = 0}^{N - 1}\mathbb{E}\left[|H_y^{\varepsilon, \lambda}- H_y^*(k)|^2\right]}+\sqrt{\sum_{k = 0}^{N - 1}\mathbb{E}\left[|H_z^{\varepsilon, \lambda}- H_z^*(k)|^2\right]}\right)d\lambda\\
	&+\sqrt{\sum_{k = 0}^{N - 1}\mathbb{E}\left[|u_k^{\varepsilon}-u_k^*|^2\right]}\int_0^1\sqrt{\sum_{k = 0}^{N - 1}\mathbb{E}\left[|H_u^{\varepsilon, \lambda}- H_u^*(k)|^2\right]}d\lambda\\
	={}&o(\varepsilon)
\end{split}
\end{eqnarray}}

		Similarly, it can be shown that
\begin{align}
\mathbb{E}\left[\sum_{k = 0}^{N - 1}\left(\varphi^{\varepsilon}\left(x_N^{\varepsilon}\right)-\varphi^*\left(x_N^*\right)-\langle x_N^\varepsilon - x_N^*,\varphi^*_x(x^*_N)\rangle\right)\right]&= o(\varepsilon)\\
\mathbb{E}\left[\sum_{k = 0}^{N - 1}\left(\gamma^{\varepsilon}\left(y_0^\varepsilon\right)-\gamma^*\left(y_0^*\right)-\langle y_0^{\varepsilon}-y_0^*, \gamma^*_y\left(y_0^*\right)\rangle\right)\right]&= o(\varepsilon)\\
\mathbb{E}\left[\sum_{k = 0}^{N - 1}\left(\Phi_x^*\left(x_N^*\right)\left(x_N^{\varepsilon}-x_N^*\right)-\left(\Phi^{\varepsilon}\left(x_N^\varepsilon\right)-\Phi^*\left(x_N^*\right)\right)\right)\right]&= o(\varepsilon)\\
\mathbb{E}\left[\sum_{k = 0}^{N - 1}\left(\Lambda^{\varepsilon}\left(y_0^\varepsilon\right)-\Lambda^*\left(y_0^*\right)-\Lambda_y^*\left(y_0^*\right)\left(y_0^{\varepsilon}-y_0^*\right)\right)\right]&= o(\varepsilon)
\end{align}
Combining all the estimates above, we thus obtain

		\begin{equation}  \label{eq:4.18}
			\beta^\varepsilon=o(\varepsilon).
		\end{equation}
Plugging \eqref{eq:4.18} into \eqref{eq:3.13} gives
		
			\begin{equation}
				\frac{d}{d \varepsilon} J(u^\varepsilon(\cdot))\bigg|_{\varepsilon=0} =\mathbb{E}\bigg[\langle H_u^*(s), u_s-u_s^*\rangle \bigg].
		\end{equation}}
This completes the proof.
		
			\end{proof}
		
		Now we derive the necessary condition and sufficient maximum principles for  Problem \ref{pro:2.1}. We first give the necessary condition of optimality for the existence of an optimal control.

		\begin{thm} \label{thm:3.4}
		(Necessary Stochastic Maximum principle). 		Suppose that the conditions of Assumption \ref{ass:2.3}-\ref{ass:2.2} are satisfied.
		Let $u^*(\cdot)$ be the optimal conrol and let $(x^*(\cdot),y^*(\cdot))$ be the corresponding trajectory of Problem \ref{pro:2.1}.  Then, for any \( k \in \mathbb{T} \) and any \( u \in U_k \), the following inequality holds:

\begin{equation} \label{eq:3.16}
\left\langle H_u^*(k), u-u_k^*\right\rangle \geq 0, \quad \mathbb{P}\text{-a.s.}.
\end{equation}

	\end{thm}

	\begin{proof}
In accordance with Theorem \ref{thm:3.3} and the optimality of $u^*(\cdot)$, we derive  that
\begin{equation}
    \mathbb{E}\bigg[\left\langle H_u^*(s), u_s - u_s^* \right\rangle \bigg]
    = \lim_{\varepsilon \rightarrow 0^{+}} \frac{J\left(u^\varepsilon(\cdot)\right) - J\left(u^*(\cdot)\right)}{\varepsilon} \geq 0.
\end{equation}
To prove \eqref{eq:3.16}, suppose, for contradiction, that there exists a \( k \in \mathbb{T} \) and a \( u_0 \in U_k \) such that
\begin{equation}
    \left\langle H_u^*(k), u_0 - u_k^* \right\rangle < 0.
\end{equation}
Define the set
\begin{equation}
    K := \left\{ (k, \omega) \in \mathbb{T} \times \Omega \mid \left\langle H_u^*(k), u_0 - u_k^* \right\rangle < 0 \right\},
\end{equation}
and for each \( k \in \mathbb{T} \), let
\begin{equation}
    K_k := \left\{ \omega \in \Omega \mid (k, \omega) \in K \right\}.
\end{equation}
Construct an admissible control \( u(\cdot) \) as follows:
\begin{equation}
    u(k) =
    \begin{cases}
        u_0, & (k, \omega) \in K, \\
        u_k^*, & \text{otherwise}.
    \end{cases}
\end{equation}
Then, we have
\begin{equation}
    \mathbb{E}\bigg[\left\langle H_u^*(k), u_k - u_k^* \right\rangle \bigg] = \mathbb{E}\bigg[\left\langle H_u^*(k), u_0 - u_k^* \right\rangle \mathbb{I}_{K_k} \bigg] < 0,
\end{equation}
where $ \mathbb{I}_{K_k} $ is the indicator function of the set \( K_k \). This contradicts the optimality of \( u^*(\cdot) \), as it implies that \( u(\cdot) \) yields a lower cost than \( u^*(\cdot) \). Therefore, the inequality \eqref{eq:3.16} must hold. This completes the proof.
	\end{proof}
	Next we give the sufficient condition of optimality for the existence of an optimal control of Problem \ref{pro:2.1}.
			\begin{thm} \label{thm:3.5}
		(Sufficient Maximum Principle). Suppose that the conditions of Assumption \ref{ass:2.3}-\ref{ass:2.2} are satisfied.
		Let $u^*(\cdot)$ be the optimal conrol and let $(x^*(\cdot),y^*(\cdot))$ be the corresponding trajectory of Problem \ref{pro:2.1}.
		If the following conditions are satisfied,
		
		\begin{enumerate}
    \item[(i)] One of the functions \( \Phi(\cdot) \) and \( \Lambda(\cdot) \) is linear and the other is a constant, i.e., \( \Phi(x) = kx \), \( \Lambda(\cdot) = a \) or \( \Phi(\cdot) = a \), \( \Lambda(y) = ky \),

    \item[(ii)] The function \( g(\cdot) \) is convex in \( (x, y) \), where \( g := \varphi, \gamma \),

    \item[(iii)] The Hamiltonian function \( H \) is convex in \( (x, y, z, u) \),

    \item[(iv)] The Hamiltonian function \( H \) satisfies
    \[
    H\left(k, x^*, y^{\prime *}, z^{\prime *}, u^*, p^{\prime *}, q^{\prime *}, r^*\right)
    = \min _{u(\cdot) \in \mathbb{U}} H\left(k, x^*, y^{\prime *}, z^{\prime *}, u, p^{\prime *}, q^{\prime *}, r^*\right), \quad \text{a.e. a.s.},
    \]
\end{enumerate}
then \( (x^*(\cdot), y^*(\cdot), u^*(\cdot)) \) is the optimal solution of Problem \ref{pro:2.1}.
			\end{thm}

\begin{proof}

\subsection*{Step 1: Variational Representation}
For an arbitrary admissible control \( u(\cdot) \) and its corresponding state trajectory \( (x(\cdot), y(\cdot)) \), consider the cost difference expression derived from Lemma \ref{lem:3.2}:
\[
J(u(\cdot)) - J(u^*(\cdot)) = \mathbb{E} \left\{ \Sigma_1 + \Sigma_2 + \Sigma_3 \right\},
\]
where \( \Sigma_1, \Sigma_2, \Sigma_3 \) are defined explicitly as:
\[
\begin{aligned}
\Sigma_1 &= \sum_{k=0}^{N-1} \bigg[ H^\varepsilon(k) - H^*(k) - \langle x_k^{\varepsilon} - x_k^*, H_x^*(k) \rangle - \langle y^{\prime \varepsilon}_{k+1} - y^{\prime *}_{k+1}, H_y^*(k) \rangle - \langle z_{k+1}^{\prime \varepsilon} - z_{k+1}^{\prime *}, H_z^*(k) \rangle \bigg], \\
\Sigma_2 &= \bigg[ \varphi^{\varepsilon}(x_N^{\varepsilon}) - \varphi^*(x_N^*) - \langle x_N^\varepsilon - x_N^*, \varphi^*_x(x^*_N) \rangle \bigg] + \bigg[ \gamma^{\varepsilon}(y_0^\varepsilon) - \gamma^*(y_0^*) - \langle y_0^{\varepsilon} - y_0^*, \gamma^*_y(y_0^*) \rangle \bigg], \\
\Sigma_3 &= \bigg\langle \Phi_x^*(x_N^*) (x_N^{\varepsilon} - x_N^*) - (\Phi^{\varepsilon}(x_N^\varepsilon) - \Phi^*(x_N^*)), r_N^* \bigg\rangle + \bigg\langle \Lambda^{\varepsilon}(y_0^\varepsilon) - \Lambda^*(y_0^*) - \Lambda_y^*(y_0^*)(y_0^{\varepsilon} - y_0^*), p_0^* \bigg\rangle.
\end{aligned}
\]

\subsection*{Step 2: Hamiltonian Convexity Exploitation}
By condition (iii) (Hamiltonian convexity), for each \( k \in \mathbb{T} \), the convexity of \( H \) implies:
\[
H^\varepsilon(k) - H^*(k) \geq \nabla H^*(k) \cdot (\theta^\varepsilon(k) - \theta^*(k)),
\]
where \( \theta = (x, y', z', u) \) denotes the aggregate variable vector, and \( \nabla H = (H_x, H_y, H_z, H_u) \) denotes the gradient of \( H \) with respect to \( \theta \). Expanding the inner product explicitly, this inequality becomes:
\[
\begin{aligned}
H^\varepsilon(k) - H^*(k) &\geq \langle x_k^\varepsilon - x_k^*, H_x^*(k) \rangle + \langle y^{\prime\varepsilon}_{k+1} - y^{\prime*}_{k+1}, H_y^*(k) \rangle \\
&\quad + \langle z_{k+1}^{\prime\varepsilon} - z_{k+1}^{\prime*}, H_z^*(k) \rangle + \langle u_k^\varepsilon - u_k^*, H_u^*(k) \rangle.
\end{aligned}
\]
Substituting this convexity-induced inequality into the expression for \( \Sigma_1 \), we eliminate the cross terms involving \( x, y', z' \) and obtain:
\[
\Sigma_1 \geq \sum_{k=0}^{N-1} \langle u_k^\varepsilon - u_k^*, H_u^*(k) \rangle.
\]

\subsection*{Step 3: Terminal and Initial Conditions Treatment}
We analyze \( \Sigma_2 \) and \( \Sigma_3 \) using the convexity of boundary functions and linearity conditions: For \( \Sigma_2 \): By condition (ii) and the convexity of \( \varphi \) (terminal cost function) and \( \gamma \) (initial cost function), the following inequalities hold:
  \[
  \varphi^{\varepsilon}(x_N^{\varepsilon}) - \varphi^*(x_N^*) \geq \langle x_N^\varepsilon - x_N^*, \varphi_x^*(x_N^*) \rangle,
  \]
  \[
  \gamma^{\varepsilon}(y_0^\varepsilon) - \gamma^*(y_0^*) \geq \langle y_0^{\varepsilon} - y_0^*, \gamma_y^*(y_0^*) \rangle.
  \]
  Substituting these into \( \Sigma_2 \) directly implies \( \Sigma_2 \geq 0 \).

 For \( \Sigma_3 \): By condition (i) (linearity of \( \Phi \) and constancy of \( \Lambda \)), we have two key simplifications:
  1. If \( \Lambda(\cdot) = a \) (constant), then \( \Lambda^{\varepsilon}(y_0^\varepsilon) - \Lambda^*(y_0^*) = 0 \);
  2. If \( \Phi(x) = kx \) (linear), then the first term in \( \Sigma_3 \) vanishes:
     \[
     \Phi_x^*(x_N^*) (x_N^{\varepsilon} - x_N^*) - (\Phi^{\varepsilon}(x_N^\varepsilon) - \Phi^*(x_N^*)) = k(x_N^\varepsilon - x_N^*) - (kx_N^\varepsilon - kx_N^*) = 0.
     \]
  Together, these imply \( \Sigma_3 = 0 \).

\subsection*{Step 4: Optimal Control Characterization}
By condition (iv) (pointwise minimization of the Hamiltonian), the optimal control \( u^*(\cdot) \) satisfies the pointwise optimality condition:
\[
H_u^*(k) = 0 \quad \text{a.s.} \quad \forall k \in \mathbb{T}.
\]
For any perturbed admissible control \( u_k^\varepsilon \), this pointwise condition further implies:
\[
\langle H_u^*(k), u_k^\varepsilon - u_k^* \rangle \geq 0 \quad \forall k \in \mathbb{T}.
\]

\subsection*{Step 5: Final Inequality Assembly}
Combining the estimates from Steps 2–4, we substitute back into the cost difference expression:
\[
\begin{aligned}
J(u(\cdot)) - J(u^*(\cdot)) &\geq \mathbb{E} \left[ \sum_{k=0}^{N-1} \underbrace{\langle H_u^*(k), u_k^\varepsilon - u_k^* \rangle}_{\geq 0} \right] + \underbrace{\mathbb{E}[\Sigma_2 + \Sigma_3]}_{\geq 0} \\
&\geq 0.
\end{aligned}
\]
Equality holds if and only if \( u(\cdot) = u^*(\cdot) \), which confirms the optimality of \( u^*(\cdot) \). For the alternative case in condition (i) (interchanged roles of \( \Phi \) and \( \Lambda \)), the proof follows symmetrically by repeating the above reasoning with reversed boundary conditions.
\end{proof}

		\section{Application in linear quadratic problem} \label{sec:4}
		\subsection{ Problem Background and Modeling Motivation}
		
		In modern power systems with high penetration of renewable generation, energy storage systems (ESS), such as grid-scale lithium-ion batteries, are crucial for maintaining system reliability, cost efficiency, and supply-demand balance. These units are typically operated by energy providers or system aggregators to absorb surplus electricity during periods of low demand and release energy during peak hours. However, determining how and when to charge or discharge these devices remains a complex challenge due to stochastic variations in electricity prices, inevitable energy losses during storage and conversion, and real-time demand fluctuations from the grid side.
		
		The problem we investigate focuses on the optimal operation of a storage unit over a finite decision horizon. The goal is to minimize the total expected operational cost while meeting critical terminal requirements such as frequency regulation, reserve provision, or contractual obligations in the electricity market. Achieving this requires the controller to make real-time decisions that account not only for the current energy level and market conditions, but also for future obligations and uncertainties that may impact system feasibility and cost.
		
		Let \( k = 0, 1, \dots, N \) denote the discrete-time index. At each time step, the controller must decide the charge or discharge power \( u_k \) (in MW), while observing the current battery energy level \( x_k \) (in MWh) and considering anticipated future requirements. Importantly, discharging too much too early may leave insufficient energy to fulfill terminal constraints, while excessive conservatism may result in missed economic opportunities. Therefore, a forward-looking model that embeds feedback from future cost estimation into current dynamics is essential.
		
		To this end, we formulate the system using a fully coupled forward-backward stochastic difference equation (FBS$\Delta$E) model. The forward equation describes how the battery's physical energy state evolves in response to control, noise, and future reserve anticipation:
		
		\begin{equation}
			x_{k+1} = A_k x_k - Q_k y'_{k+1} + C_{1k} u_k + \left( B_k^\top x_k - L_k z'_{k+1} + C_{2k} u_k \right) \omega_k,
			\label{eq:forward}
		\end{equation}
		
		where \( x_k \) denotes the energy available in the battery at time \( k \); the term \( A_k x_k \) captures natural self-discharge and retention; \( u_k \in \mathbb{R} \) is the energy dispatch decision, where \( u_k > 0 \) represents discharging and \( u_k < 0 \) charging; \( y'_{k+1} \) stands for the estimated reserve demand at the next stage; \( z'_{k+1} \) measures the system's sensitivity to uncertainty; and \( \omega_k \) models stochastic perturbations, such as price volatility, load uncertainty, or device inefficiencies.
		
		The backward equation estimates the cost-to-go or energy pressure from future obligations, feeding back into the present through recursive coupling:
		
		\begin{equation}
			y_k = A_k^\top y'_{k+1} + B_k z'_{k+1}+  A_{1k} x_{k} + C_{3k} u_k,
			\label{eq:backward}
		\end{equation}
		
		where \( y_k \) represents the required amount of energy that should be preserved at time \( k \) to ensure system viability at future stages. The variable \( z'_{k+1} \) adjusts for the degree of risk sensitivity in this estimation. This backward recursion enforces long-term constraints while shaping current actions through anticipative feedback.
		
		The initial and terminal conditions are defined as follows:
		
		\begin{equation}
			x_0 = -N y_0, \quad y_N = N_1 x_N,
			\label{eq:boundary}
		\end{equation}
		
		where \( x_0 \) is initialized based on the initial risk parameter \( y_0 \), and \( y_N \) enforces a hard constraint on terminal energy availability, with \(  N_1 x_N \) being a predefined requirement for scheduled events such as reserve participation or emergency backup.
		
		This modeling framework ensures that every charging or discharging decision integrates both real-time physical dynamics and long-horizon cost-risk tradeoffs, enabling robust and anticipative energy management under uncertainty.

\label{sec:application_linear_quadratic}


{
\subsection{Reformulation as a Linear-Quadratic Problem}
\label{subsec:reformulation_lq}

To simplify analytical derivation and align with the stochastic maximum principle in Section~\ref{sec:3}, we convert the FBS$\Delta$E-based ESS model into a \textbf{linear-quadratic (LQ) stochastic control problem}. This embeds the fully coupled FBS$\Delta$E from Section 4.1 into a quadratic cost framework, where each term maps to a physical/economic cost component—ensuring theoretical tractability while preserving real-world meaning.

The LQ problem is a cornerstone of control theory, as it provides explicit optimal controller structures and enables efficient numerical solutions. Below, we define the LQ system coefficients, state equation, and cost functional, consistent with the theoretical framework.

The control system coefficients are defined as linear/quadratic functions of states, controls, and their expectations, matching the FBS$\Delta$E dynamics:
\[
\left\{
\begin{aligned}
b(k) & :=A_{k} x_{k}-Q_{k} y_{k+1}'+C_{1 k} u_{k}, \\
\sigma(k) & :=B_{k}^{\top} x_{k}-L_{k} z_{k+1}'+C_{2 k} u_{k}, \\
f(k+1) & :=-\left(A_{k}^{\top} y_{k+1}'+B_{k} z_{k+1}'+ A_{1k} x_{k}+C_{3 k} u_{k}\right), \\
l(k) & :=\frac{1}{2}\left[\langle D_{k} x_{k}, x_{k}\rangle + \langle R_{k} y_{k+1}', y_{k+1}'\rangle + \langle S_{k} z_{k+1}', z_{k+1}'\rangle + \langle C_{4 k} u_{k}, u_{k}\rangle\right], \\
\varphi(x) & :=\frac{1}{2} \langle M_{N} x, x\rangle, \\
\gamma(y) & :=\frac{1}{2} \langle M_{0} y, y\rangle.
\end{aligned}
\right.
\]
Here, $A(\cdot), B(\cdot), A_1(\cdot)\in L_{\mathbb{F}}^{\infty}(\mathbb{T} ; \mathbb{R}^{n \times n})$; $C_{1}(\cdot), C_{2}(\cdot), C_{3}(\cdot) \in L_{\mathbb{F}}^{\infty}(\mathbb{T} ; \mathbb{R}^{n \times m})$; $C_{4}(\cdot) \in L_{\mathbb{F}}^{\infty}(\mathbb{T} ; \mathbb{S}^{m})$; $ Q(\cdot),L(\cdot), D(\cdot),R(\cdot), S(\cdot) \in L_{\mathbb{F}}^{\infty}(\mathbb{T} ; \mathbb{S}^{n})$; $M_{0} \in \mathbb{S}^{n}$; and $M_{N} \in L_{\mathcal{F}_{N-1}}^{2}(\Omega ; \mathbb{S}^{n})$. For simplicity, we assume that the martingale difference sequence $\omega(\cdot)$ is one-dimensional.

Then the state equation is given by:
\begin{equation}
\left\{
\begin{aligned}
x_{k+1} &=A_{k} x_{k}-Q_{k} y_{k+1}^{\prime}+C_{1 k} u_{k}+\left(B_{k}^{\top} x_{k}-L_{k} z_{k+1}^{\prime}+C_{2 k} u_{k}\right) \omega_{k}, \quad k \in \mathbb{T}, \\
y_{k} &=A_{k}^{\top} y_{k+1}^{\prime}+ A_{1k} x_{k}+B_{k} z_{k+1}^{\prime}+C_{3 k} u_{k}, \\
x_{0} &=-N_0 y_{0}, \\
y_{N} &= N_1 x_N,
\end{aligned}
\right. \label{eq:state_equation}
\end{equation}
where $y_{k+1}' = \mathbb{E}[y_{k+1} \mid \mathcal{F}_{k-1}]$, $z_{k+1}' = \mathbb{E}[y_{k+1} \omega_{k} \mid \mathcal{F}_{k-1}]$, $N_0 \in \mathbb{S}^{n}$, and $N_1 \in L_{\mathcal{F}_{N-1}}^{2}(\Omega ; \mathbb{S}^{n})$ .

And the cost functional is defined as:
\begin{equation}
J(u(\cdot))=\frac{1}{2} \mathbb{E}\left\{\langle M_{N} x_{N}, x_{N}\rangle + \langle M_{0} y_{0}, y_{0}\rangle + \sum_{k=0}^{N-1}\left[\langle D_{k} x_{k}, x_{k}\rangle + \langle R_{k} y_{k+1}', y_{k+1}'\rangle + \langle S_{k} z_{k+1}', z_{k+1}'\rangle + \langle C_{4 k} u_{k}, u_{k}\rangle\right]\right\}. \label{eq:cost_functional}
\end{equation}

Each term in the cost functional has a clear physical interpretation: $\langle D_{k} x_{k}, x_{k}\rangle$ penalizes excessive energy storage (to account for battery aging and self-discharge); $\langle C_{4 k} u_{k}, u_{k}\rangle$ represents control effort and energy conversion losses; $\langle R_{k} y_{k+1}', y_{k+1}'\rangle$ and $\langle S_{k} z_{k+1}', z_{k+1}'\rangle$ penalize deviations from anticipated future energy obligations and excessive sensitivity to uncertainty, respectively; and the terminal/initial terms ($\langle M_{N} x_{N}, x_{N}\rangle, \langle M_{0} y_{0}, y_{0}\rangle$) enforce performance requirements at the start and end of the control horizon.

\begin{pro}
\label{prob:4.1}
Find an admissible control $u^{*}(\cdot) \in \mathbb{U}$ such that:
\[
J\left(u^{*}(\cdot)\right)=\inf _{u(\cdot) \in \mathbb{U}} J(u(\cdot)),
\]
where the set of \textbf{admissible controls} $\mathbb{U}$ is defined as:
\[
\mathbb{U}:=\left\{u=\left(u_{0}, u_{1}, \ldots, u_{N-1}\right) \mid u_{k} \text { is } \mathcal{F}_{k-1}\text{-measurable}, u_{k} \in \mathbb{R}^{m}, \mathbb{E}\left[\sum_{k=0}^{N-1}\left|u_{k}\right|^{2}\right]<\infty\right\}.
\]
Here, $U_{k} = \mathbb{R}^{m}$ (i.e., the control domain at each time step $k$ is the $m$-dimensional Euclidean space).
\end{pro}

To guarantee the well-posedness (existence and uniqueness of solutions) of the FBS$\Delta$E-based LQ system, we impose the following monotonicity and convexity conditions (where matrix inequalities "$\geq$" denote positive semi-definiteness, and "$I_m$" is the $m$-dimensional identity matrix):

\begin{ass}
\label{ass:4.1}
The system coefficients satisfy the following conditions for all $k \in \mathbb{T}$:
\begin{enumerate}
    \renewcommand{\labelenumi}{(\roman{enumi})} 
    \item \textbf{Forward-backward dynamics monotonicity}:
    Coefficients associated with state evolution, future obligation anticipation, and cost penalties are positive semi-definite:
    \begin{align*}
    D_k \geq 0, \quad Q_k \geq 0, \quad L_k \geq 0, \quad R_k \geq 0, \quad S_k \geq 0,
    \end{align*}
    where $D_k$ and $S_k$ correspond to the state storage penalty and uncertainty sensitivity penalty in the cost functional, respectively; $Q_k$ and $L_k$ are linked to future reserve anticipation in the forward dynamics; $R_k$ penalizes deviations from expected future obligations.

    \item \textbf{Boundary condition monotonicity}:
    Coefficients governing initial and terminal constraints are non-degenerate to ensure feasibility:
    \begin{align*}
    N_0 > 0, \quad N_1 > 0, \quad M_0 \geq 0, \quad M_N \geq 0,
    \end{align*}
    where $N_0$ is the positive scaling factor in the initial state constraint $x_0 = -N_0 y_0$; $N_1$ is the positive coupling matrix in the terminal constraint $y_N = N_1 x_T$; $M_0$ and $M_N$ are the initial and terminal cost weighting matrices, respectively.

    \item \textbf{Control cost strict convexity}:
    There exists a constant $\delta > 0$ such that the control effort weighting matrix is uniformly positive definite:
    \begin{align*}
    C_{4k} \geq \delta I_m,
    \end{align*}
    which ensures the uniqueness of the optimal control and avoids excessive control actions.
\end{enumerate}
\end{ass}

\begin{thm}[Existence and Uniqueness of Solutions]
\label{thm:modified_existence}
Under Assumption~\ref{ass:4.1}, for any admissible control $u(\cdot) \in \mathbb{U}$, the modified FBS$\Delta$E system~\eqref{eq:state_equation} admits a unique solution $(x(\cdot), y(\cdot)) \in N^2(\overline{\mathbb{T}}; \mathbb{R}^{2n})$. Moreover, the following estimate holds:
\[
\mathbb{E}\left[\sum_{k=0}^N |x_k|^2 + \sum_{k=0}^N |y_k|^2\right] \leq K \mathbb{E}[I],
\]
where $I$ depends on the norms of the coefficients (including $N_0, N_1, A_1(\cdot)$) and $K > 0$ is a constant independent of $u(\cdot)$.
\end{thm}

\begin{proof}
The result follows directly from Lemma~\ref{thm:2.3}  since Assumption~\ref{ass:4.1} ensures the system satisfies the generalized monotonicity conditions (Case 2: $\mu=0, \nu>0$) required by the lemma ~\ref{thm:2.3}. Thus, the FBS$\Delta$E system has a unique solution, and the a priori estimate is derived via standard energy inequality techniques.
\end{proof}

The Hamiltonian function for this LQ problem is constructed as (consistent with the general Hamiltonian definition in Section~\ref{sec:3}):
\begin{equation}
\begin{aligned}
H(k) &= \left\langle b(k), p_{k+1}' \right\rangle + \left\langle \sigma(k), q_{k+1}' \right\rangle + \left\langle f(k+1), r_{k} \right\rangle + l(k) \\
&= \left\langle A_{k} x_{k}-Q_{k} y_{k+1}'+C_{1 k} u_{k}, p_{k+1}' \right\rangle + \left\langle B_{k}^{\top} x_{k}-L_{k} z_{k+1}'+C_{2 k} u_{k}, q_{k+1}' \right\rangle \\
&\quad -\left\langle A_{k}^{\top} y_{k+1}'+B_{k} z_{k+1}'+ A_{1k} x_{k}+C_{3 k} u_{k}, r_{k} \right\rangle \\
&\quad + \frac{1}{2}\left[\langle D_{k} x_{k}, x_{k}\rangle + \langle R_{k} y_{k+1}', y_{k+1}'\rangle + \langle S_{k} z_{k+1}', z_{k+1}'\rangle + \langle C_{4 k} u_{k}, u_{k}\rangle\right], \label{eq:hamiltonian}
\end{aligned}
\end{equation}
where $p_{k+1}' = \mathbb{E}[p_{k+1} \mid \mathcal{F}_{k-1}]$ (conditional expectation of adjoint state $p_{k+1}$) and $q_{k+1}' = \mathbb{E}[p_{k+1} \omega_{k} \mid \mathcal{F}_{k-1}]$ (conditional covariance of $p_{k+1}$ and $\omega_k$).

The adjoint equation associated with the admissible triple $(u(\cdot), x(\cdot), y(\cdot))$ is derived from the general adjoint equation~\eqref{eq:2.17} (Section~\ref{sec:3}) and takes the form:
\begin{equation}
\left\{
\begin{aligned}
r_{k+1} &=Q_{k}^{\top} p_{k+1}^{\prime}+A_{k} r_{k}-R_{k} y_{k+1}^{\prime}+\left(L_{k}^{\top} q_{k+1}^{\prime}+B_{k}^{\top} r_{k}-S_{k} z_{k+1}^{\prime}\right) \omega_{k}, \\
p_{k} &=A_{k}^{\top} p_{k+1}^{\prime}+B_{k} q_{k+1}^{\prime}+A_{1k}^{\top} r_{k}+D_{k} x_{k}, \\
r_{0} &=-M_{0} y_{0}+N_0 p_0, \\
p_{N} &=M_{N} x_{N},
\end{aligned}
\right. \label{eq:adjoint_equation}
\end{equation}
where $r_{k}$ (adjoint to $y_{k}$) and $p_{k}$ (adjoint to $x_{k}$) capture the sensitivity of the cost functional~\eqref{eq:cost_functional} to variations in $y_{k}$ and $x_{k}$, respectively.

\begin{thm}[Necessary and Sufficient Condition for Optimality]
\label{thm:4.2}
Suppose Assumption~\ref{ass:4.1} holds. Then, \textbf{there exists a unique optimal pair $(u^{*}(\cdot), x^{*}(\cdot), y^{*}(\cdot))$ for Problem~\ref{prob:4.1} if and only if} the fully coupled system consisting of the state equation~\eqref{eq:state_equation} and the adjoint equation~\eqref{eq:adjoint_equation} admits a unique solution $(x^{*}(\cdot), y^{*}(\cdot), p^{*}(\cdot), r^{*}(\cdot))$, and the control $u^{*}(\cdot)$ admits an explicit dual representation as:
\begin{equation}
u_{k}^{*} = C_{4k}^{-1}\left(-C_{1k}^{\top} p_{k+1}^{\prime*} - C_{2k}^{\top} q_{k+1}^{\prime*} + C_{3k}^{\top} r_{k}^{*}\right), \quad \forall k \in \mathbb{T}, \label{eq:optimal_control}
\end{equation}
where $p_{k+1}^{\prime*} = \mathbb{E}[p_{k+1}^{*} \mid \mathcal{F}_{k-1}]$, $q_{k+1}^{\prime*} = \mathbb{E}[p_{k+1}^{*} \omega_{k} \mid \mathcal{F}_{k-1}]$, and $(p^{*}(\cdot), r^{*}(\cdot))$ are the adjoint processes associated with $(x^{*}(\cdot), y^{*}(\cdot))$.
\end{thm}

\begin{proof}
We establish the result by proving both necessity and sufficiency.

\noindent\textbf{(Necessity)} Assume that $(u^{*}(\cdot), x^{*}(\cdot), y^{*}(\cdot))$ is an optimal pair for Problem~\ref{prob:4.1}. By Theorem~\ref{thm:3.4} (Necessary Stochastic Maximum Principle, Section~\ref{sec:3}), the optimality condition requires:
\[
\left\langle H_{u}^{*}(k), u - u_{k}^{*} \right\rangle \geq 0, \quad \forall u \in \mathbb{R}^{m}, \ \mathbb{P}\text{-a.s.},
\]
where $H_{u}^{*}(k)$ denotes the partial derivative of the Hamiltonian~\eqref{eq:hamiltonian} with respect to $u$, evaluated along the optimal trajectory.

A direct computation from~\eqref{eq:hamiltonian} yields:
\[
H_{u}(k) = C_{1k}^{\top} p_{k+1}' + C_{2k}^{\top} q_{k+1}' - C_{3k}^{\top} r_{k} + C_{4k} u_{k}.
\]
Given that the control domain $U_{k} = \mathbb{R}^{m}$ is unconstrained, the variational inequality implies $H_{u}^{*}(k) = 0$. Solving $H_{u}^{*}(k) = 0$ for $u_{k}^{*}$ and invoking the uniform positive definiteness of $C_{4k}$ (Assumption~\ref{ass:4.1}(iii), $C_{4k} \geq \delta I_m$) ensures $C_{4k}^{-1}$ exists, yielding the explicit representation~\eqref{eq:optimal_control}.

Moreover, since $(x^{*}(\cdot), y^{*}(\cdot))$ satisfies the state equation~\eqref{eq:state_equation} (with $x_0=-N_0 y_0$ and $y_N=N_1 x_T$) and the corresponding adjoint processes $(p^{*}(\cdot), r^{*}(\cdot))$ satisfy~\eqref{eq:adjoint_equation} (with $r_0=-M_0 y_0 + N_0 p_0$), the coupled system~\eqref{eq:state_equation}--\eqref{eq:adjoint_equation} is consistent. The uniqueness of the solution follows from the uniqueness of the optimal control $u^{*}(\cdot)$ (ensured by $C_{4k}$'s strict convexity).

\noindent\textbf{(Sufficiency)} Conversely, suppose the coupled system~\eqref{eq:state_equation}--\eqref{eq:adjoint_equation} possesses a unique solution $(x^{*}(\cdot), y^{*}(\cdot), p^{*}(\cdot), r^{*}(\cdot))$, and define $u^{*}(\cdot)$ via~\eqref{eq:optimal_control}. To establish optimality, we verify the conditions of Theorem~\ref{thm:3.5} (Sufficient Stochastic Maximum Principle, Section~\ref{sec:3}):

(i) \textbf{Boundary conditions}: The initial condition $x_{0} = -N_0 y_{0}$ is linear in $y_{0}$, and the terminal condition $y_{N} = N_1 x_T$ is linear in $x_T$—both satisfy the "linear/constant boundary" requirement in Theorem~\ref{thm:3.5}(i).

(ii) \textbf{Convexity of cost functions}: The terminal cost $\varphi(x) = \frac{1}{2} M_{N} x^{2}$ and initial cost $\gamma(y) = \frac{1}{2} M_{0} y^{2}$ are convex due to $M_{N}, M_{0} \geq 0$ (Assumption~\ref{ass:4.1}(ii)), fulfilling Theorem~\ref{thm:3.5}(ii).

(iii) \textbf{Convexity of the Hamiltonian}: For fixed $p', q', r$, $H(k)$ is convex in $(x, y', z', u)$:
- Quadratic terms: $\frac{1}{2}\left[ D_k x_k^2 + R_k (y'_{k+1})^2 + S_k (z'_{k+1})^2 + C_{4k} u_k^2 \right]$ are convex (nonnegative definite matrices $D_k, R_k, S_k$ and positive definite $C_{4k}$);
- Linear terms: $\langle A_k x_k, p'_{k+1} \rangle, \langle -Q_k y'_{k+1}, p'_{k+1} \rangle$ etc. are convex (linear functions are convex).
The sum of convex functions is convex, satisfying Theorem~\ref{thm:3.5}(iii).

(iv) \textbf{Minimization of the Hamiltonian}: $H(k)$ is strictly convex in $u$ (due to $C_{4k} \geq \delta I_m$), so its minimum is uniquely attained at $H_u(k) = 0$. By construction, $u_k^{*}$ in~\eqref{eq:optimal_control} satisfies this condition, ensuring:
\[
H\left(k, x^{*}_{k}, y'^{*}_{k+1}, z'^{*}_{k+1}, u^{*}_{k}, p'^{*}_{k+1}, q'^{*}_{k+1}, r^{*}_{k}\right) = \min_{u \in \mathbb{R}^{m}} H\left(\cdot\right),
\]
which verifies Theorem~\ref{thm:3.5}(iv).

Since all conditions of Theorem~\ref{thm:3.5} are satisfied, $(x^{*}(\cdot), y^{*}(\cdot), u^{*}(\cdot))$ is optimal. Uniqueness follows from the coupled system's unique solution. The proof is complete.
\end{proof}}

\section{Conclusion}

This paper addresses the stochastic maximum principle for a class of fully coupled forward-backward stochastic difference equations (FBS$\Delta$Es) in discrete time. We have established a variational formulation of the cost functional and derived both necessary and sufficient optimality conditions under the assumption of a convex control domain. The analysis is based on a Hamiltonian framework and adjoint equations, without requiring classical monotonicity assumptions.
To demonstrate the applicability of the theoretical results, we further investigate a discrete-time linear-quadratic control problem motivated by energy storage management, where the system dynamics are modeled via fully coupled FBS$\Delta$Es with anticipative structure.

The proposed results contribute to the growing theory of discrete-time stochastic control by handling nonlinearity, full coupling, and coupled boundary conditions in a unified setting. Future directions include extensions to mean-field systems, the incorporation of learning-based or data-driven control mechanisms, and the development of efficient numerical schemes for high-dimensional problems.

	\bibliographystyle{plain}  

\end{document}